\documentclass[preprint,12pt,authoryear]{elsarticle}

\usepackage{amssymb}
\usepackage{amsmath}
\usepackage{amsthm}
\usepackage{tikz-cd}
\usepackage{ebproof}
\usepackage{cleveref}

\newtheorem{theorem}{Theorem}[section]
\newtheorem*{theorem*}{Theorem}

\newtheorem{lemma}[theorem]{Lemma}
\newtheorem{corollary}[theorem]{Corollary}

\newtheorem{proposition}[theorem]{Proposition}

\theoremstyle{definition}
\newtheorem{example}[theorem]{Example}
\newtheorem{definition}[theorem]{Definition}
\newtheorem{remark}[theorem]{Remark}

\AddToHook{env/fact/begin}{\crefalias{theorem}{fact}}
\AddToHook{env/lemma/begin}{\crefalias{theorem}{lemma}}
\AddToHook{env/corollary/begin}{\crefalias{theorem}{corollary}}
\AddToHook{env/claim/begin}{\crefalias{theorem}{claim}}
\AddToHook{env/proposition/begin}{\crefalias{theorem}{proposition}}
\AddToHook{env/example/begin}{\crefalias{theorem}{example}}
\AddToHook{env/definition/begin}{\crefalias{theorem}{definition}}
\AddToHook{env/remark/begin}{\crefalias{theorem}{remark}}

\crefname{theorem}{theorem}{theorems}
\crefformat{theorem}{#2theorem~#1#3} 
\Crefformat{theorem}{#2Theorem~#1#3}

\crefname{fact}{fact}{facts}
\crefformat{fact}{#2fact~#1#3} 
\Crefformat{fact}{#2Fact~#1#3}

\crefname{lemma}{lemma}{lemmas}
\crefformat{lemma}{#2lemma~#1#3} 
\Crefformat{lemma}{#2Lemma~#1#3} 

\crefname{corollary}{corollary}{corollarys}
\crefformat{corollary}{#2corollary~#1#3} 
\Crefformat{corollary}{#2Corollary~#1#3} 

\crefname{claim}{claim}{claims}
\crefformat{claim}{#2claim~#1#3} 
\Crefformat{claim}{#2Claim~#1#3} 

\crefname{proposition}{proposition}{propositions}
\crefformat{proposition}{#2proposition~#1#3} 
\Crefformat{proposition}{#2Proposition~#1#3}

\crefname{example}{example}{examples}
\crefformat{example}{#2example~#1#3} 
\Crefformat{example}{#2Example~#1#3} 

\crefname{definition}{definition}{definitions}
\crefformat{definition}{#2definition~#1#3} 
\Crefformat{definition}{#2Definition~#1#3}

\crefname{remark}{remark}{remarks}
\crefformat{remark}{#2remark~#1#3} 
\Crefformat{remark}{#2Remark~#1#3}

\newcommand{\mc}[1]{\mathcal{#1}}

\newcommand{\mb}[1]{\mathbf{#1}}
\newcommand{\mbb}[1]{\mathbb{#1}}
\newcommand{\ms}[1]{\mathsf{#1}}

\newcommand{\mr}[1]{\mathrm{#1}}

\newcommand{\pa}{\ms{PA}}
\newcommand{\N}{\mbb N}
\newcommand{\T}{\mbb T}
\newcommand{\buq}[2]{\forall #1\!\!<\!\!#2\,}
\newcommand{\beq}[2]{\exists #1\!\!<\!\!#2\,}
\newcommand{\muq}[2]{\mu #1\!\!<\!\!#2\,}

\newcommand{\rec}{\ms{rec}}
\newcommand{\rem}{\ms{rem}}

\newcommand{\PriM}{\mb{PriM}}
\newcommand{\pri}{\ms{pair}}
\newcommand{\fst}{\ms{fst}}
\newcommand{\snd}{\ms{snd}}
\newcommand{\lh}{\ms{lh}}

\newcommand{\nt}{\Rightarrow}
\newcommand{\Set}{\mb{Set}}

\newcommand{\set}[1]{\left\{#1\right\}}

\newcommand{\cv}{\downarrow}

\newcommand{\sub}{\operatorname{Sub}}

\newcommand{\inj}{\rightarrowtail}
\newcommand{\surj}{\twoheadrightarrow}
\newcommand{\id}{\operatorname{id}}
\newcommand{\pair}[1]{\langle #1 \rangle}

\newcommand{\eff}{\Leftrightarrow}

\newcommand{\ov}[1]{\overline{#1}}

\newcommand{\qsi}[1]{\widetilde{#1}}

\newcommand{\pf}[1]{\widehat{#1}}
\newcommand{\inv}{^{-1}}
\newcommand{\ev}{\mathrm{ev}}

\newcommand{\hook}{\hookrightarrow}

\newcommand{\other}{\text{otherwise}}

\journal{Theoretical Computer Science}

\begin{document}

\begin{frontmatter}

\title{Categorical structure in coherent theory of arithmetic} 

\author{Lingyuan Ye} 

\affiliation{organization={University of Cambridge},
            city={Cambridge},
            country={UK}}

\begin{abstract}
In this paper we provide a semantic and syntactic analysis of parametrised natural numbers object in coherent categories, or pr-coherent categories. Semantically, we show the definable functions in the initial pr-coherent category are exactly given by primitive recursive functions. We also show that any pr-coherent category supports the construction of bounded universal quantifications, which are absent in an arbitrary coherent category. Under these semantic consideration, we construct a coherent theory of arithmetic and we show its syntactic category is equivalent to the initial pr-coherent category. From a logical perspective, we also show that this theory can be identified as the $\Sigma_1$-fragment of $I\Sigma_1$. Thus as an application, we provide a structural proof of the classical result in proof theory that the strongly $\Sigma_1$-representable functions in $I\Sigma_1$ are exactly primitive recursive functions.
\end{abstract}




\end{frontmatter}


\section{Introduction}\label{sec:intro}

There have been plenty of attempts to provide a categorical foundation of computability theory in the literature. For instance, it was suggested by Lambek in various lectures and talks that the ``natural'' recursion theories and their corresponding classes of computable functions should be linked to the computable numerical functions represented in various \emph{free} categorical structures with natural numbers object (NNO); e.g. see~\citet{hofstra2021aspects} for an overview. One important consequence of this perspective is that it naturally links computability to logic, because these free categories often can be presented by various logical or type theories. 

We already have many examples of this form. For instance, the classical monograph~\citep{lambek1988introduction} contains the classification of definable functions in the free Cartesian closed category with an NNO, the free topos with an NNO, etc. However, considerably fewer efforts have been made in weaker categorical contexts. One important exception is the result in~\citet{roman1989cartesian}, which classifies definable functions in the free \emph{Cartesian} categories with an NNO, which are exactly primitive recursive functions. 

This paper is concerned with a similar question in the context of \emph{coherent categories}. In particular, we consider coherent categories with a parametrised natural numbers object (PNO), or in short \emph{pr-coherent categories}. Since coherent categories are not Cartesian closed, hence without higher types, the stronger notion of a \emph{parametrised} natural numbers object makes sure that the primitive recursion scheme internal to such categories can have arbitrary parameters, which is not guaranteed simply by an NNO; cf.~\citet[A2.5]{johnstone2002sketches}. The main contributions of this paper are twofold. 

\subsection{Definable functions in the initial pr-coherent category}
Firstly, from a semantic perspective, we have classified the definable functions in the initial pr-coherent category, which we show are again primitive recursive (\Cref{thm:defninprcoh}). Comparing this with the result of~\citet{roman1989cartesian}, we come to the conclusion that the additional logical apparatus present in coherent categories, i.e. the existence of disjunction and existential quantification of predicates, do not enlarge the class of definable functions. This is particularly surprising when taking into consideration that the recursion principle associated to the PNO in a coherent category is also stronger, because induction holds for \emph{all coherent formulas} in this case, i.e. formulas with connectives $(\top,\wedge,\bot,\vee,\exists)$, rather than just equalities between terms. In contrast, if the underlying category has full first-order structures, i.e. it is a Heyting category or a Boolean category, then the class of definable functions between natural numbers object will be a \emph{proper extension} of primitive recursive functions; see e.g.~\citet{pudlak2017,moerdijk1997minimal}.

One of the main technical achievements in this aspect is the construction of a pr-coherent category $\PriM$ of \emph{recursively enumerable sets} (\Cref{def:prim}). Concretely, we have generalised the notion of primitive recursive functions to recursively enumerable sets, rather than just powers of the natural numbers. We then show that recursively enumerable sets with this notion of primitive recursive functions between them form a pr-coherent category, and the forgetful functor $U : \PriM \to \Set$ preserves the pr-coherent structures (\Cref{prop:primprcoherent}). This way, for the initial pr-coherent category $\mc C$ with the canonical interpretation $\mc C \to \Set$, initiality implies that this interpretation functor must factor through $\PriM$:
\[
  \begin{tikzcd}
    \mc C \ar[dr] \ar[rr] & & \Set \\ 
    & \PriM \ar[ur] &
  \end{tikzcd}
\]
As a consequence, definable functions for the PNO in the initial pr-coherent category are contained in $\PriM$, which means they are all primitive recursive. 

\subsection{Presentation of the initial pr-coherent category by an arithmetic theory}
The second contribution of this paper can be considered syntactic, which explains in great detail the relation between the above semantic result and arithmetic theories in logic. As mentioned, oftentimes a free category can be represented as the \emph{syntactic category} of a logical theory. We show the initial pr-coherent category is presented by the $\Sigma_1$-fragment of $I\Sigma_1$.

$I\Sigma_1$ is a classical arithmetic theory, which is weaker than Peano arithmetic $\pa$, obtained by restricting the induction axioms in $\pa$ to $\Sigma_1$-formulas only. An arithmetic formula is $\Sigma_1$ if it is constructed out of atomic formulas using the connectives in $(\top,\wedge,\bot,\vee,\exists,\buq x{\ms t})$. In particular, $\buq x{\ms t}$ is usually referred to as the \emph{bounded universal quantifier}, where $x$ is a variable and $\ms t$ is a term not containing the variable $x$. The existence of bounded universal quantification is an important difference between $\Sigma_1$ and coherent formulas.

Thus, one main technical difficulty in establishing the link between the $\Sigma_1$-fragment of $I\Sigma_1$ and the initial pr-coherent category is that, \emph{a priori}, coherent categories do not support the interpretation of universal quantifiers. Our contribution here is to show in any pr-coherent category, the relevant bounded universal quantifiers \emph{do exist} (\Cref{cor:sigmabuq}).

In fact, to present the initial pr-coherent category, our strategy is to directly construct a coherent theory of arithmetic $\T$ with bounded universal quantifiers (\Cref{def:coharith}). We will show the syntactic category $\mc C[\T]$ of $\T$ is the initial pr-coherent category (\Cref{thm:ctinitial}). The initiality result means that our axiomatisation of $\T$ captures precisely the semantic content of pr-coherent categories. In particular, any pr-coherent category supports the interpretation of $\T$, and these interpretations will be preserved under coherent functors preserving the PNO (\Cref{lem:stability}).

Later we also establish that the syntactic category $\mc C[I\Sigma_1]_{\Sigma_1}$ of the $\Sigma_1$-fragment of $I\Sigma_1$ is equivalent to $\mc C[\T]$ (\Cref{cor:tsigonefrag}). However, if one is not motivated in traditional proof theory, then maybe the more important point to focus on is the existence of \emph{some} coherent theory of arithmetic that represents the initial pr-coherent category. Thus, we will mainly work with $\T$ in this paper, hopefully to increase readability.

\subsection{Application of initiality}\label{subsec:application}

Equipped with the semantic and syntactic result in this paper, the following theorem follows as a consequence:

\begin{theorem*}[Strongly representable functions in $I\Sigma_1$]\label{thm:isigma}
  In $I\Sigma_1$, the \emph{strongly $\Sigma_1$-representable functions} are exactly the primitive recursive functions.
\end{theorem*}

We say a function $f : \N^k \to \N$ is strongly $\Sigma_1$-representable in $I\Sigma_1$, if there is a $\Sigma_1$-formula $\varphi_f(\ov x,y)$ that satisfies the following two conditions:\footnote{In this paper we will generally use the over bar to denote a list of objects.}
\begin{itemize}
\item $\varphi_f(\ov x,y)$ defines the graph of $f$: For any $n_1,\cdots,n_k,m\in\N$,
  \[ \N \models \varphi_f(\ov{\ms n},\ms m) \eff f(\ov n) = m, \]
  where $\ms n_1,\cdots,\ms n_k,\ms m$ are \emph{numerals} of $n_1,\cdots,n_k,m$ in the theory.
\item $\varphi_f$ is \emph{provably functional} in $I\Sigma_1$,
  \[ I\Sigma_1 \vdash \forall\ov x\exists_{!}y\varphi_f(\ov x,y). \]
\end{itemize}
As we will see, the strongly $\Sigma_1$-representable functions in $I\Sigma_1$ exactly correspond to morphisms in $\mc C[I\Sigma_1]_{\Sigma_1}$. Thus, these functions are definable in the initial pr-coherent category, which are primitive recursive.

In the traditional proof-theoretic context, strongly $\Sigma_1$-representable functions are also referred to as \emph{provably total recursive functions}. The idea is that the graph of every recursive function is recursively enumerable, and every recursively enumerable sets can be defined by a $\Sigma_1$-formula. Thus, the set of strongly $\Sigma_1$-representable functions reflects the \emph{strength} of the arithmetic theory under consideration, by characterising the class of total recursive functions whose totality can be verified by this theory. The above theorem is one of the early cornerstones of the proof-theoretic analysis of arithmetic.

In fact, one of the initiatives of this paper is to provide a \emph{structural} proof of this result. The usual proof-theoretic strategy relies on very detailed analysis of different sequent calculi representing $I\Sigma_1$. Our categorical approach instead provides a proof \emph{invariant} under the proof system in use. In particular, the analysis in this paper shows that the validity of this statement comes from the following structural results:
\begin{itemize}
\item The $\Sigma_1$-fragment of $I\Sigma_1$ presents the initial pr-coherent category;
\item Primitive recursive functions between recursively enumerable sets form a pr-coherent category.
\end{itemize}

In the end, we will also discuss some further applications of initiality of $\T$ in \Cref{sec:furtherp} via its Freyd cover, a.k.a. Artin glueing; cf.~\citet{freyd1972aspects,Moerdijk1983}. Most importantly, it turns out that $\T$-provability coincides with \emph{truth} in the standard model (\Cref{thm:trpreq}). This implies e.g. that $\T$ has the \emph{disjunction property} (\Cref{cor:disjct}) and the \emph{existence property} (\Cref{cor:cant}). We end the paper by providing a summary of what we have done and indicate some future directions in \Cref{sec:future}.

\section{Coherent arithmetic}\label{sec:coharith}

In this section, we construct a coherent theory of arithmetic $\T$ that will ultimately be shown to present the initial pr-coherent category. Before defining $\T$, it might be better to first fully specify what \emph{is} a coherent theory. We do this quite carefully because part of the aim of this section is to establish some elementary arithmetic in $\T$, which involves working with its proof system. For general references on coherent logic, we refer the reader to \citet[D1.3]{johnstone2002sketches} and \citet[Sec. 1.2]{caramello2018theories}.

Given a vocabulary $\mc L$ which specifies the sorts, function symbols, and relation symbols, coherent $\mc L$-formulas are those first-order formulas that can be constructed by these symbols together with $=,\bot,\top,\wedge,\vee,\exists$. The deduction system of coherent theory involves \emph{sequents} of coherent formulas
\[ \varphi \vdash_{\ov x} \psi. \]
Here the list of variables $\ov x$ will be called the \emph{context} of the sequent, and the free variables of $\varphi$ and $\psi$ must be contained in $\ov x$.

The deduction system of coherent logic first contains the following \emph{structural rules}:
\begin{itemize}
\item \emph{Identity}:
  \[
    \begin{prooftree}
      \hypo{}
      \infer 1[(Id)]{\varphi \vdash_{\ov x} \varphi}
    \end{prooftree}
  \]
\item \emph{Substitution}: For any list of terms $\ov t$ in context $\ov y$,
  \[
    \begin{prooftree}
      \hypo{\varphi \vdash_{\ov x} \psi}
      \infer 1[(Sub)]{\varphi[\ov t/\ov x] \vdash_{\ov y} \psi[\ov t/\ov x]}
    \end{prooftree}
  \]
  Notice that the above rule only requires $\ov y$ to \emph{contain} free variables in $\ov t$, hence both \emph{weakening} and \emph{contraction} of contexts are special cases. \vspace{.5ex}
\item \emph{Cut}:
  \[
    \begin{prooftree}
      \hypo{\varphi \vdash_{\ov x} \psi}
      \hypo{\psi \vdash_{\ov x} \chi}
      \infer 2[(Cut)]{\varphi \vdash_{\ov x} \chi}
    \end{prooftree}
  \]
\end{itemize}\vspace{2ex}
It also contains the following axioms and rules for logical connectives:
\begin{itemize}
\item \emph{Conjunction}: Three axioms
  \[
    \begin{prooftree}
      \hypo{}
      \infer 1{\varphi \vdash_{\ov x} \top}
    \end{prooftree} \quad
    \begin{prooftree}
      \hypo{}
      \infer 1{\varphi \wedge \psi \vdash_{\ov x} \varphi}
    \end{prooftree} \quad
    \begin{prooftree}
      \hypo{}
      \infer 1{\varphi \wedge \psi \vdash_{\ov x} \psi}
    \end{prooftree}
  \]
  plus a rule
  \[
    \begin{prooftree}
      \hypo{\chi \vdash_{\ov x} \varphi}
      \hypo{\chi \vdash_{\ov x} \psi}
      \infer 2{\chi \vdash_{\ov x} \varphi \wedge \psi}
    \end{prooftree}
  \]
\item \emph{Disjunction}: Three axioms
  \[
    \begin{prooftree}
      \hypo{}
      \infer 1{\bot \vdash_{\ov x} \varphi}
    \end{prooftree} \quad
    \begin{prooftree}
      \hypo{}
      \infer 1{\varphi \vdash_{\ov x} \varphi \vee \psi}
    \end{prooftree} \quad
    \begin{prooftree}
      \hypo{}
      \infer 1{\psi \vdash_{\ov x} \varphi \vee \psi}
    \end{prooftree}
  \]
  plus a rule
  \[
    \begin{prooftree}
      \hypo{\varphi \vdash_{\ov x} \chi}
      \hypo{\psi \vdash_{\ov x} \chi}
      \infer 2{\varphi \vee \psi \vdash_{\ov x} \chi}
    \end{prooftree}
  \]
\item \emph{Existential Quantifier}: A double rule
  \[
    \begin{prooftree}
      \hypo{\varphi \vdash_{\ov x,y} \psi}
      \infer[double]1{\exists y\varphi \vdash_{\ov x} \psi}
    \end{prooftree}
  \]
  By our specification of contexts, $y$ must not occur freely in $\psi$. \vspace{.5ex}
\item \emph{Distributivity} and \emph{Frobenius}: Two more axioms,
  \[
    \begin{prooftree}
      \hypo{}
      \infer 1{\varphi \wedge (\psi \vee \chi) \vdash_{\ov x} (\varphi \wedge \chi) \vee (\varphi \vee \chi)}
    \end{prooftree} \quad
    \begin{prooftree}
      \hypo{}
      \infer 1{\varphi \wedge \exists y\psi \vdash_{\ov x} \exists y(\varphi \wedge \psi)}
    \end{prooftree}
  \]
  Given the context, on the right $y$ must not occur freely in $\varphi$. \vspace{.5ex}
\item \emph{Identity}: An axiom
  \[
    \begin{prooftree}
      \hypo{}
      \infer 1{\top \vdash_{x} x = x}
    \end{prooftree}
  \]
  and a double rule
  \[
    \begin{prooftree}
      \hypo{\varphi[y/z] \vdash_{\ov x,y} \psi[y/z]}
      \infer[double]1{\varphi \wedge y = z \vdash_{\ov x,y,z} \psi}
    \end{prooftree}
  \]
\end{itemize}

Of course, the above deduction system is not designed optimally for proof-theoretic analysis, but instead for actual reasoning in coherent logic. The crucial point is that these axioms and rules are naturally sound and complete with respect to the internal logic of coherent categories; cf.~\citet[D1.5]{johnstone2002sketches}.

Now we may proceed to construct the coherent theory $\T$ of arithmetic. Since firstly we need it to be able to construct all primitive recursive functions, we might as well add all of them as \emph{terms}. This is strictly speaking not necessary, since the induction rule with a few basic functions would already guarantee the existence of all primitive recursive functions. However, adding them explicitly will simplify the process of defining coding in $\T$, which will be useful later in \Cref{sec:init}.

We first use mutual recursion to define a set \emph{PrimRec} of function descriptions of primitive recursive functions, together with an arity function $\mr{ar} : \textit{PrimRec} \to \N$, and a denotational semantics $\mathrm{ev}$, sending each function description to the function it denotes:
\begin{definition}[Function description]\label{def:funcdesc}
  The set \emph{PrimRec}, an \emph{arity function} $\mr{ar} : \textit{PrimRec} \to \N$, and its denotational semantics $\mathrm{ev}$ are given mutually recursively by:
  \begin{itemize}
  \item $\ms z,\ms s$ and $\pi_n^k$ for all $1 \le k \le n$ belong to \emph{PrimRec}, with their arity being $1,1$ and $n$, respectively, and
    \[ \mathrm{ev}(\ms z)(x) = 0, \quad \mathrm{ev}(\ms s)(x) = x+1, \quad \mathrm{ev}(\pi_n^k)(x_1,\cdots,x_n) = x_k. \]
  \item If $\ms h$ and $\ms g_1,\cdots,\ms g_n$ are in \emph{PrimRec}, with
    \[ \mr{ar}(\ms h) = n,\ \mathrm{ev}(\ms h) = h, \quad \forall i \le n.\ \mr{ar}(\ms g_i) = m,\ \mathrm{ev}(\ms g_i) = g_i, \]
    then $\mathrm{Cn}[\ms h,\ms g_1,\cdots,\ms g_n] \in \textit{PrimRec}$, with $\mr{ar}(\mathrm{Cn}[\ms h,\ms g_1,\cdots,\ms g_n]) = m$, and $f = \ev(\mathrm{Cn}[\ms h,\ms g_1,\cdots,\ms g_n])$ is defined by composition
    \begin{align*}
      f(x_1,\cdots,x_m) = h(g_1(x_1,\cdots,x_m),\cdots,g_n(x_1,\cdots,x_m)). 
    \end{align*}
  \item If $\ms g$ and $\ms h$ are in \emph{PrimRec}, with
    \[ \mr{ar}(\ms g) = n,\ \mathrm{ev}(\ms g) = g, \quad \mr{ar}(\ms h) = n+2,\ \mathrm{ev}(\ms h) = h, \]
    then $\mathrm{Pr}[\ms g,\ms h]\in$ \emph{PrimRec}, with $\mr{ar}(\mathrm{Pr}[\ms g,\ms h]) = n+1$, and $f = \ev(\mathrm{Pr}[\ms g,\ms h])$ is defined by primitive recursion
    \begin{align*}
      f(x_1,\cdots,x_n,y) =
      \begin{cases}
        g(x_1,\cdots,x_n) & y = 0 \\
        h(x_1,\cdots,x_n,z,f(x_1,\cdots,x_n,z)) & y = z + 1
      \end{cases}
    \end{align*}
  \end{itemize}
\end{definition}

Our idea is then to add all elements in \emph{PrimRec} as function symbols to our theory $\T$, together with their defining axioms. 

\begin{definition}[Coherent Arithmetic]\label{def:coharith}
  The theory of coherent arithmetic $\T$ has a constant $\ms 0$, a set \emph{PrimRec} of function symbols with their arity specified by $\mr{ar}$, and a binary predicate $<$ for the ``less than'' relation. $\T$ contains the following axioms:
  \begin{itemize}
  \item For $\ms z$ and $\pi_n^k$ for any $1 \le k \le n$, we have
    \[ \top \vdash_x \ms zx = \ms 0, \quad \top \vdash_{x_1,\cdots,x_n} \pi_n^k(x_1,\cdots,x_n) = x_k. \]
  \item For $\ms s$ we have
    \[ \ms sx = \ms 0 \vdash_x \bot, \quad \ms sx = \ms sy \vdash_{x,y} x = y. \]
  \item If $\ms f$ is of the form $\mathrm{Cn}[\ms h,\ms g_1,\cdots,\ms g_n]$, then we have
    \[ \top \vdash_{\ov x} \ms f(\ov x) = \ms h(\ms g_1(\ov x),\cdots,\ms g_n(\ov x)). \]
  \item If $\ms f$ is of the form $\mathrm{Pr}[\ms g,\ms h]$, then we have
    \[ \top \vdash_{\ov x} \ms f(\ov x,\ms 0) = \ms g(\ov x), \quad \top \vdash_{\ov x,y} \ms f(\ov x,\ms sy) = \ms h(\ov x,y,\ms f(\ov x,y)). \]
  \item For the binary predicate $<$, we have
    \[ x < y \dashv\vdash_{x,y} \exists z(x + \ms sz = y), \]
    where $+$ is the function symbol in \emph{PrimRec} corresponding to the usual primitive recursive definition of addition.
  \end{itemize}
  Besides the usual logical connectives and their axioms of coherent logic, we further equip $\T$ with the following additional structures:
  \begin{itemize}
  \item \emph{Bounded universal quantification}: For any $\T$-formula $\varphi(\ov x)$, $\buq z{t(\ov y)}\varphi$ is also a $\T$-formula for a $\T$-term $t$ that does not contain $z$, with the following double rule:
    \[
      \begin{prooftree}
        \hypo{\psi(\ov x,\ov y) \vdash_{\ov x,\ov y} \buq z{t}\varphi(\ov x,\ov y,z)}
        \infer[double]1{\psi(\ov x,\ov y) \wedge z < t \vdash_{\ov x,\ov y,z} \varphi(\ov x,\ov y,z)}
      \end{prooftree}
    \]
    Note the context specifies $z$ must not occur in $\psi$ as a free variable. \vspace{.5ex}
  \item \emph{Induction rule}: It has a right induction rule
    \[
      \begin{prooftree}
        \hypo{\psi(\ov x) \vdash_{\ov x} \varphi(\ov x,\ms 0)}
        \hypo{\psi(\ov x) \wedge \varphi(\ov x,y) \vdash_{\ov x,y} \varphi(\ov x,\ms sy)}
        \infer 2[(IndR)]{\psi(\ov x) \vdash_{\ov x,y} \varphi(\ov x,y)}
      \end{prooftree}
    \]
  \end{itemize}
\end{definition}

\begin{remark}[$\T$ is not a coherent theory \emph{a priori}]
  From the above construction, $\T$ fails to be coherent: its syntax allows for a new constructor of formulas by bounded universal quantification, and it extends coherent logic by a new induction \emph{rule}. Formulas in $\T$ more precisely should be referred to as $\Sigma_1$-formulas, rather than coherent formulas. However, we emphasise that if one only aims for the construction of a logical theory that represents the initial pr-coherent category, then the important things to check are simply (1) the syntactic category $\mc C[\T]$ of $\T$ will be a coherent category, which is indeed the case since $\T$ contains the connectives and rules for coherent logic (\Cref{prop:cohsyn}); and (2) the logical mechanisms in $\T$ are present in all pr-coherent categories, which will be shown in \Cref{sec:pno}. Whether $\T$ is strictly a coherent theory is not important from this perspective.
\end{remark}

\begin{remark}[$\T$ is a coherent theory \emph{after all}]\label{rem:coherentT}
  From another perspective, the semantic investigation of pr-coherent categories in \Cref{sec:pno} can be used to construct $\T$ as an \emph{honest} coherent theory; cf. \Cref{rem:tiscoherent}. The point is that every $\Sigma_1$-formula, modulo $\T$, will be equivalent to a coherent formula. Furthermore, we can replace adding the induction rule by the set of coherent sequents which are provable in $\T$. This way, we get an equivalent coherent axiomatisation of $\T$, and in the future we will still refer to $\T$ as a coherent theory. The reason that we favour our construction of $\T$ is for clarity and simplicity.
\end{remark}

To have a sense of what it is like to work within the theory $\T$, we show some examples of derivations in this system. We take this chance to establish some elementary arithmetic in $\T$ that will be used later when proving the initiality of $\mc C[\T]$. Through these examples, we intend to convince the reader that $\T$, though lacks the full apparatus of classical first-order logic, is sufficient to develop all basic arithmetic.

Since a potential audience might not be familiar with coherent logic in the first place, we start by establishing some logical principles which are admissible in coherent logic:

\begin{example}[Explicit existence implies existence]
  For formula $\varphi(\ov x,y)$ and term $t(\ov x)$, the principle $\varphi(\ov x,t(\ov x)) \vdash_{\ov x} \exists y\varphi(\ov x,y)$ is indeed provable. A formal derivation is as follows,
  \[
    \begin{prooftree}
      \hypo{}
      \infer 1[(Id)]{\exists y\varphi(\ov x,y) \vdash_{\ov x} \exists y\varphi(\ov x,y)}
      \infer 1{\varphi(\ov x,y) \vdash_{\ov x,y} \exists y\varphi(\ov x,y)}
      \infer 1[(Sub)]{\varphi(\ov x,t(\ov x)) \vdash_{\ov x} \exists y\varphi(\ov x,y)}
    \end{prooftree}
  \]
  The second step uses the double rule for existential quantifier, and the last step uses the structure rule of substitution. Notice in particular the context contraction from the second sequent to the last one. This is allowed by the definition of the substitution rule.
\end{example}

\begin{example}[Equality is a congruence]
  Given the rule for equality, we can deduce symmetry and transitivity of equality as follows:
  \[
    \begin{prooftree}
      \hypo{}
      \infer 1{\top \vdash_x x = x}
      \infer 1{x = y \vdash_{x,y} y = x}
    \end{prooftree}
  \]
  This seems mysterious at first. Let $\equiv$ denote the \emph{syntactic} equality of formulas, viz. equality of strings. What we are really doing is defining two formulas $\varphi :\equiv \top$ and $\psi :\equiv y = x$, hence $\varphi[x/y] \equiv \top$, and $\psi[x/y] \equiv x = x$. The above derivation then follows from the equality rule specified above. Similarly, for transitivity we also have
  \[
    \begin{prooftree}
      \hypo{}
      \infer 1{x = z \vdash_{x,z} x = z}
      \infer 1{x = y \wedge z = y \vdash_{x,y,z} x = z}
    \end{prooftree}
  \]
  Again, we have chosen $\varphi :\equiv x = y$ and $\psi :\equiv x = z$, then $\varphi[z/y] \equiv x = z$ and $\psi[z/y] \equiv x = z$. Modulo symmetry, this indeed gives us transitivity of equality. These properties plus the substitution rule are sufficient to show that equality is a congruence for all the functions and predicates in a theory.
\end{example}

Other valid logical principles that are familiar to us are also derivable in any coherent theory, which we leave for the reader to check. Next, we develop some basic arithmetic within $\T$, and we will freely use the afore-proven results. We will also revert to the usual informal argument when discussing provability within a formal theory, instead of explicitly writing out the full derivation tree every time.

\begin{example}[Successor]
  One basic principle is that every number is either zero or a successor, and we can prove $\top \vdash_x x = \ms 0 \vee \exists y(x = \ms sy)$ in $\T$. The strategy is to apply the (IndR) rule:
  \begin{itemize}
  \item For the base case, $\top \vdash \ms 0 = \ms 0$, hence also $\top \vdash \ms 0 = \ms 0 \vee \exists y(\ms 0 = \ms sy)$.
  \item For induction, $x = \ms 0 \vdash_x \ms sx = \ms s\ms 0$, thus $x = \ms 0 \vdash_x \exists y(\ms sx = \ms sy)$. We also have $x = \ms sy \vdash_x \ms sx = \ms s\ms sy$, thus it implies $x = \ms sy \vdash_{x,y} \exists y(\ms sx = \ms sy)$, and then it follows $\exists y(x = \ms sy) \vdash_x \exists y(\ms sx = \ms sy)$.
  \end{itemize}
\end{example}

\begin{example}[Addition]
  There is a binary function symbol in \emph{PrimRec} corresponding to the usual primitive recursive definition of addition. Thus in $\T$, there exists a function $+$, which we write in the usual infix notation, with the following axiomatisation,
  \[ \top \vdash_x x + \ms 0 = x, \quad \top \vdash_{x,y} x + \ms sy = \ms s(x + y). \]
  Equipped with the induction rule, all familiar elementary properties of addition are provable. For instance, we can show $\top \vdash_x \ms 0 + x = x$ in $\T$:
  \begin{itemize}
  \item Base case: Trivially, we have $\top \vdash \ms 0 + \ms 0 = \ms 0$ provable from our axiom. 
  \item Inductive case: The following can be reasoned within $\T$,
    \[ \ms 0 + x = x \vdash_x \ms 0 + \ms sx = \ms s(\ms 0 + x) = \ms sx. \]
  \end{itemize}
  As another example, we can show $\top \vdash_{x,y} \ms sx + y = \ms s(x + y)$ is provable in $\T$:
  \begin{itemize}
  \item Base case: Trivially, $\top \vdash_x \ms sx + \ms 0 = \ms sx$.
  \item Inductive case: Again we reason in $\T$ as follows,
    \[ \ms sx + y = \ms s(x + y) \vdash_{x,y} \ms sx + \ms sy = \ms s(\ms sx + y) = \ms s\ms s(x + y) = \ms s(x + \ms sy). \]
  \end{itemize}
  Hence by the induction principle, again $\top \vdash_{x,y} \ms sx + y = \ms s(x + y)$ is derivable. Using the above two facts, we can show as usual that addition is associative and commutative. The same for associativity and commutativity of multiplication, and the distributivity of multiplication over addition.
\end{example}

\begin{example}[Subtraction]
  Firstly, there is a unary function symbol $\ms p$ corresponding to the \emph{predecessor} function, with the axiomatisation in $\T$
  \[ \top \vdash \ms p\ms 0 = \ms 0, \quad \top \vdash_x \ms p(\ms sx) = x. \]
  Based on this, there is also a function symbol in \emph{PrimRec} corresponding to the primitive recursive definition of (truncated) subtraction. In $\T$, we denote this function symbol simply as $-$, with the following axiomatisation,
  \[ \top \vdash_x x - \ms 0 = x, \quad \top \vdash_{x,y} x - \ms sy = \ms p(x - y). \]
  We first show by induction on $y$ that $\top \vdash_{x,y} x - y = \ms sx - \ms sy$:
  \begin{itemize}
  \item For the base case, $\top \vdash_x x - \ms 0 = x$, and also
    \[ \top \vdash_x \ms sx - \ms s\ms 0 = \ms p(\ms sx - \ms s0) = \ms p\ms sx = x, \]
    thus $\top \vdash_x x - \ms 0 = \ms sx - \ms s\ms 0$.
  \item For the inductive case, we reason in $\T$
    \[ x - y = \ms sx - \ms sy \vdash_{x,y} x - \ms sy = \ms p(x - y) = \ms p(\ms sx - \ms sy) = \ms sx - \ms s\ms sy. \]
  \end{itemize}
  Using this, we can show that $\top \vdash_x x - x = \ms 0$:
  \begin{itemize}
  \item The base case is easy, since by axiom $\top \vdash \ms 0 - \ms 0 = \ms 0$;
  \item For the inductive case, we have $x - x = \ms 0 \vdash_x \ms sx - \ms sx = x - x = \ms 0$.
  \end{itemize}
  Similarly, by a double induction on $y,z$ we can also show in $\T$ that
  \[ \top \vdash_{x,y,z} (x + y) - z = (x - z) + y. \]
  We leave this for the reader to check.
\end{example}

\begin{example}[Order]\label{exm:order}
  From the properties of addition and subtraction, we can show in $\T$ the usual properties of the order relation $<$, viz. its transitivity, anti-symmetry, and linearity. Transitivity is evident from the commutativity and associativity of addition. For the other properties, we first observe that $x + y = x \vdash_{x,y} y = \ms 0$:
  \begin{align*}
    x + y = x
    &\vdash_{x,y} (x + y) - x = x - x \\
    &\vdash_{x,y} (x - x) + y = \ms 0 \\
    &\vdash_{x,y} \ms 0 + y = \ms 0 \\
    &\vdash_{x,y} y = \ms 0.
  \end{align*}
  Here we have used the previously established properties of subtraction. Now we can prove by induction that anti-symmetry $x < y \wedge y < x \vdash_{x,y} \bot$ holds:
  \begin{align*}
    x + \ms sz = y \wedge y + \ms sw = x
    &\vdash_{x,y,z,w} y + \ms sw + \ms sz = y \\
    &\vdash_{x,y,z,w} y + \ms s(w + \ms sz) = y \\
    &\vdash_{x,y,z,w} \ms s(w + \ms sz) = \ms 0 \\
    &\vdash_{x,y,z,w} \bot
  \end{align*}
  The above reasoning has used associativity of addition. By Frobenius and the rule for existential quantification, we get
  \[ \exists z(x + \ms sz = y) \wedge \exists w(y + \ms sw = x) \vdash_{x,y} \bot, \]
  or equivalently,
  \[ x < y \wedge y < x \vdash_{x,y} \bot. \]
  Similarly, we can also show $x < y \wedge x = y \vdash_{x,y} \bot$ in $\T$:
  \begin{align*}
    x + \ms sz = y \wedge_{x,y,z} x = y
    &\vdash_{x,y,z} x + \ms sz = x \\
    &\vdash_{x,y,z} \ms sz = \ms 0 \\
    &\vdash_{x,y,z} \bot.
  \end{align*}
  This again allows us to conclude
  \[ x < y \wedge x = y \vdash_{x,y} \bot. \]
  These facts give us half of trichotomy. For the other half, we need to show
  \[ \top \vdash_{x,y} x < y \vee x = y \vee y < x. \]
  We show by induction on $y$. The base case is trivial, since when $y = \ms 0$, we already have
    \[ \top \vdash_x x = \ms 0 \vee \exists z(x = \ms sz), \]
    and this implies $\top \vdash_x x = \ms 0 \vee \ms 0 < x$. For the induction step, we need to show $x < y \vee x = y \vee y < x \vdash_{x,y} x < \ms sy \vee x = \ms sy \vee \ms sy < x$. We distinguish three cases:
  \begin{enumerate}
  \item $x < y$: This trivially implies $x < \ms sy$.
  \item $x = y$: This also trivially implies $x < \ms sy$.
  \item $y < x$: We show $y + \ms sz = x \vdash_{x,y,z} x = \ms sy \vee \ms sy < x$. Further distinguish two cases. Either $z = \ms 0$, then $x = \ms sy$, or $\exists u(z = \ms su)$, then we have
    \begin{align*}
      y + \ms s\ms su = x
      &\vdash \ms s(y + \ms su) = x \\
      &\vdash \ms sy + \ms su = x
    \end{align*}
    This way, $\ms sy < x$ holds.
  \end{enumerate}
  This completes the proof. In the future, we also abbreviate $x < y \vee x = y$ as $x \le y$. The above results show that in $\T$ we can verify that $<,\le$ behave as expected.
\end{example}

At this point, we hope we have convinced the reader that all elementary arithmetic can be developed inside $\T$. At the end of this section, we use these to show that $\T$ also admits another induction rule. When doing elementary arithmetic, sometimes one would also want to do induction on a negated formula. However, since $\T$ is a coherent theory hence lack negation, this is not possible directly. But we do have a version of this expressed by the following \emph{left} induction rule:

\begin{proposition}\label{prop:leftindrule}
  The following \emph{left induction rule} is admissible in $\T$,
  \[
    \begin{prooftree}
      \hypo{\psi(\ov x,\ms 0) \vdash_{\ov x} \varphi(\ov x)}
      \hypo{\psi(\ov x,\ms sy) \vdash_{\ov x,y} \varphi(\ov x) \vee \psi(\ov x,y)}
      \infer 2[$(\mathrm{IndL})$]{\psi(\ov x,y) \vdash_{\ov x,y} \varphi(\ov x)}
    \end{prooftree}
  \]
\end{proposition}
\begin{proof}
  Suppose we have $\psi(\ov x,\ms 0) \vdash_{\ov x} \varphi(\ov x)$ and $\psi(\ov x,\ms sy) \vdash_{\ov x,y} \varphi(\ov x) \vee \psi(\ov x,y)$. We need to show the conclusion holds. Firstly, it is easy to see that
  \[ \ms 0 < x \vdash_x \ms s\ms p(x) = x, \]
  because $\ms 0 < x$ implies $x = \ms sy$ for some $y$. Similarly, we can also show
  \[ y < z \vdash_{y,z} \ms s(z - sy) = z - y. \]
  Now consider a new predicate
  \[ \psi'(\ov x,y,z) := \psi(\ov x,z-y). \]
  This way, by assumption we have
  \[ y < z \wedge \psi'(\ov x,y,z) \vdash_{\ov x,y,z} \varphi(\ov x) \vee \psi'(\ov x,\ms s y,z). \]
  We prove by induction on $y$ that
  \[ \ms 0 < z \wedge \psi'(\ov x,\ms 0,z) \vdash_{\ov x,y,z} z \le y \vee \varphi(\ov x) \vee \psi'(x,\ms sy,z). \]
  \begin{itemize}
  \item The base case $y = \ms 0$ is easy, since by assumption $y < z$, and when this holds we do have $\psi'(\ov x,y,z) \vdash_{\ov x,y,z} \varphi(\ov x) \vee \psi'(\ov x,\ms sy,z)$.
  \item For the inductive case, it suffices to show the following holds,
    \[ z \le y \vee \psi'(\ov x,y,z) \vdash_{\ov x,y,z} z \le \ms sy \vee \varphi(\ov x) \vee \psi'(\ov x,\ms sy,z). \]
    If $z \le y$, then evidently $z \le \ms sy$. If $y < z$, then from above we know that $y < z \wedge \psi'(\ov x,y,z)$ implies $\varphi(\ov x) \vee \psi'(\ov x,\ms sy,z)$.
  \end{itemize}
  In particular, this would imply that
  \[ \ms 0 < z \wedge \psi'(\ov x,\ms 0,z) \vdash_{\ov x,z} \varphi(\ov x) \vee \psi'(x,z,z), \]
  because if $0 < z$ then there exists some $u<z$ that $z = \ms su$. Thus, according to the definition of $\psi'$, we obtain
  \[ \ms 0 < z \wedge \psi(\ov x,z) \vdash_{\ov x,z} \varphi(\ov x) \vee \psi(\ov x,\ms 0). \]
  We then conclude $\psi(\ov x,y) \vdash_{\ov x,y} \varphi(\ov x)$: If $y = \ms 0$, we already have $\psi(\ov x,\ms 0) \vdash_{\ov x} \varphi(\ov x)$ by assumption; if $\ms 0 < y$, we also have $\psi(\ov x,y) \vdash_{\ov x,y} \varphi(\ov x) \vee \psi(\ov x,\ms 0) \vdash_{\ov x,y} \varphi(x)$. This completes the proof.
\end{proof}

\section{Parametrised natural numbers object in coherent categories}\label{sec:pno}

In this section, we provide a semantic analysis of pr-coherent categories, viz. coherent categories equipped with a parametrised natural numbers object. In particular, we will show that the logical mechanisms present in our coherent theory of arithmetic $\T$ introduced in the previous section is also available in any pr-coherent category, and furthermore preserved by any pr-coherent functor, viz. coherent functors preserving the PNO.

A coherent category is a category $\mc C$ that has (a) finite limits, (b) universal image factorisation of morphisms, and (c) universal finite joins in the subobject lattice; universality means that the corresponding structure is preserved by pullbacks. We refer the reader to~\citet[A1.4]{johnstone2002sketches} for definition and properties of coherent categories. For us, the most important fact is that any coherent category has an \emph{internal logic} that validates all the axioms and rules in coherent logic. In this section, we freely use the internal logic as internal constructions in $\mc C$. 

As mentioned in the introduction, the reason we consider \emph{parametrised} natural numbers object, rather than simply the natural numbers object, is that $\mc C$ in general will \emph{not} be Cartesian closed, thus lacks higher types. Then a natural numbers object will not be able to support recursion with parameters; cf.~\citet[A2.5]{johnstone2002sketches}. The notion of PNO is defined as follows:

\begin{definition}[Parametrised natural numbers object]
  An object $N$ in $\mc C$ is a \emph{parametrised natural numbers object}, or a \emph{PNO}, if it is equipped with,
  \[ \ms 0 : 1 \to N, \quad \ms s : N \to N, \]
  such that for any $f : A \to X$ and $g : X \to X$ in $\mc C$, there is a \emph{unique} map $\rec_{f,g} : A \times N \to X$ that makes the following diagram commute,
  \[
  \begin{tikzcd}
    A \ar[dr, "f"'] \ar[r, "\pair{\id,\ms 0}"] & A \times N \ar[d, "{\rec_{f,g}}"] & A \times N \ar[l, "{\id \times\ms s}"'] \ar[d, "{\rec_{f,g}}"] \\
    & X & X \ar[l, "g"]
  \end{tikzcd}
  \]
\end{definition}

\subsection{Primitive recursion and order}

It is well-known that if $\mc C$ has a PNO $N$, then we can construct all the primitive recursive functions as morphisms between powers of $N$ internally in $\mc C$; cf.~\citet[A2.5]{johnstone2002sketches}. For convenience of the reader, we reproduce the argument here:

\begin{lemma}\label{lem:pnoprimc}
  Let $N$ be a PNO in $\mc C$. For any $g : A \to B$, $h : A \times N \times B \to B$, there exists a unique map $f : A \times N \to B$ such that,
  \[
  \begin{tikzcd}
    A \ar[r, "{\pair{\id,\ms 0}}"] \ar[dr, "g"'] & A \times N \ar[d, "f"] & A \times N \ar[l, "{\id \times\ms s}"'] \ar[d, "{\pair{\id,f}}"] \\
    & B & A \times N \times B \ar[l, "h"]
  \end{tikzcd}
  \]
  Intuitively, $f$ is obtained from primitive recursion as follows\footnote{Since here $A$ and $B$ are arbitrary objects, and in particular can be products of $N$, we don't lose any generality of primitive recursion.}
  \[
    f(a,n) =
    \begin{cases}
      g(a) & n = 0 \\
      h(a,m,f(a,m)) & n = m + 1
    \end{cases}
  \]
\end{lemma}
\begin{proof}
  Since $N$ is a PNO, we consider the following induced diagram,
  \[
  \begin{tikzcd}
    A \ar[r, "{\pair{\id,\ms 0}}"] \ar[ddr, bend right, "{\pair{id,\ms 0}}"'] \ar[dr, "{\pair{\id,\ms 0,g}}"'] & A \times N \ar[d, "{\rec}"] & A \times N \ar[l, "{\id \times\ms s}"'] \ar[d, "{\rec}"] \\
    & A \times N \times B \ar[d, "{\pi_{A \times N}}"] & A \times N \times B \ar[l, "{\pair{\pi_A,\ms s \circ \pi_N,h}}"] \ar[d, "{\pi_{A \times N}}"] \\
    & A \times N & A \times N \ar[l, "{\id \times \ms s}"]
  \end{tikzcd}
  \]
  where the various maps labelled by $\pi$ are projections. By uniqueness of the universal property of PNO, we must have
  \[ \pi_{A \times N} \circ \rec = \id. \]
  Thus, $\rec$ must be of the form $\pair{\id,f}$, for some $f : A \times N \to B$. Now commutativity of the upper level gives us the desired property.
\end{proof}

This way, the PNO $N$ in $\mc C$ admits an interpretation of all the function symbols in $\T$, because all of them are generated from $\ms z,\ms s,\pi_n^k$ via composition and primitive recursion, which exist for $N$ in $\mc C$. Similarly, the order $<$ with the corresponding definition as
\[ x < y := \exists z(x + \ms sz = y) \]
is also a definable subobject of $N \times N$ in $\mc C$, because the above formula is coherent. It follows that the PNO in $\mc C$ can interpret all the non-logical symbols of the coherent arithmetic $\T$.

\subsection{Induction principle}
The next step is to show that the \emph{induction rule} is also valid in the internal logic of $\mc C$ with the PNO $N$. To formulate the induction principle, consider an object $X$ in $\mc C$. Given two subobjects $\varphi,\psi$, we write
\[ X \models \varphi(x) \vdash \psi(x) \]
if $\varphi \le \psi$ in the subobject lattice $\sub(X)$. We also write
\[ X \models \varphi(x) \dashv\vdash \psi(x), \]
if $\varphi$ and $\psi$ agree in $\sub(X)$.

The reason we attach the variable $x$ to these subobjects above is because it provides a syntactic way of using \emph{substitutions} to denote pullbacks. For instance, given a subobject $\varphi$ of $X$ and a morphism $f : Y \to X$, we will use $\varphi(f(y))$ to denote the following pullback,
\[
  \begin{tikzcd}
    \varphi(f(y)) \ar[r] \ar[d, hook] & \varphi(x) \ar[d, hook] \\
    Y \ar[r, "f"'] & X
    \arrow["\lrcorner"{anchor=center, pos=0.125}, draw=none, from=1-1, to=2-2]
  \end{tikzcd}
\]
The reason to write them in \emph{sequent} form is because these semantic relations are \emph{closed} under derivation rules of coherent logic (cf. \Cref{sec:coharith}), which are expressed in sequent form. In particular, the logical rules for $\top,\bot,\wedge,\vee,\exists$ are \emph{sound} for the internal logic of $\mc C$; cf.~\citet[D1.3.2]{johnstone2002sketches}.

Similarly, the validity of the induction principle for $N$ in $\mc C$ should also mean that these semantic relations will be closed under application of (IndR). If we can show this, then all the syntactic proofs we have given in \Cref{sec:coharith} also become true in the internal logic of $\mc C$.

\begin{proposition}[PNO satisfies the right induction rule]\label{lem:indpno}
  Let $N$ be a PNO in $\mc C$. For any subobject $\varphi$ of $X \times N$ and $\psi$ of $X$ in $\mc C$, if we have
  \[ X \models \psi(x) \vdash \varphi(x,\ms 0), \quad X \times N \models \psi(x) \wedge \varphi(x,n) \vdash \varphi(x,\ms s n), \]
  then we also have
  \[ X \times N \models \psi(x) \vdash \varphi(x,n). \]
\end{proposition}
\begin{proof}
  Notice that we have two commutative squares as follows,
  \[
    \begin{tikzcd}
      \psi \ar[d, hook] \ar[r] & \varphi \ar[d, hook] \\
      X \ar[r, "{\pair{\id,\ms 0}}"'] & X \times N
    \end{tikzcd} \quad
    \begin{tikzcd}
      \psi \ar[d, hook] \ar[r, "{\pair{\id,\ms 0}}"] & \psi \times N \ar[d, hook] \\
      X \ar[r, "{\pair{\id,\ms 0}}"'] & X \times N
    \end{tikzcd}
  \]
  The left square commutes simply because $X \models \psi(x) \vdash \varphi(x,\ms 0)$, and the right square commutes simply by the product structure. This shows that the two squares define the same map from $\psi$ to $X \times N$, hence by the universal property of pullback, we get a map of the following type,
  \[ \psi \to (\psi \times N) \wedge \varphi, \]
  where the conjunction $\wedge$ is taken in the subobject lattice $\sub(X \times N)$. On the other hand, $X \times N \models \psi(x) \wedge \varphi(x,n) \vdash \varphi(x,\ms sn)$ implies we have
  \[
    \begin{tikzcd}
      (\psi \times N) \wedge \varphi \ar[dr, hook] \ar[r, hook] & \varphi(x,\ms sn) \ar[d, hook] \ar[r, "\ms s"] & \varphi(x,n) \ar[d, hook] \\
      & X \times N \ar[r, "\id \times \ms s"'] & X \times N 
    \end{tikzcd}
  \]
  This shows that $\ms s$ induces an \emph{endomorphism} on the subobject $(\psi \times N) \wedge \varphi$, essentially by taking $(x,n)$ in this subobject to $(x,\ms sn)$. We again denote this map as $\id \times \ms s$. Now we can use the universal property of PNO, to create a following diagram,
  \[
    \begin{tikzcd}
      \psi \ar[r, "{\pair{\id,0}}"] \ar[dr] \ar[ddr, bend right] & \psi \times N \ar[d, "\rec"] & \psi \times N \ar[l, "\ms s"'] \ar[d, "\rec"] \\
      & (\psi \times N) \wedge \varphi \ar[d, hook] & (\psi \times N) \wedge \varphi \ar[l, "\id \times \ms s"'] \ar[d, hook] \\
      & \psi \times N & \psi \times N \ar[l, "\id \times \ms s"']
    \end{tikzcd}
  \]
  By the universal property, the composition of $\rec$ and the inclusion above must be identity. This shows that
  \[ X \times N \models \psi(x) \vdash \varphi(x,n). \qedhere \]
\end{proof}

\begin{remark}
  Notice that \Cref{lem:indpno} is in fact stronger than merely saying that the induction rule in $\T$ is valid for $N$ in $\mc C$. \emph{A priori}, the induction rule in $\T$ only applies to $\Sigma_1$-formulas, while \Cref{lem:indpno} applies to arbitrary predicates in $\mc C$, which may not be $\Sigma_1$. For instance, if $\mc C$ is a Boolean category, then \Cref{lem:indpno} would imply that the induction rule holds for all classical predicates on $N$ in $\mc C$ as well.
\end{remark}

Now if we recall the proof of admissibility of the left induction rule (IndL) in $\T$ given in Proposition~\Ref{prop:leftindrule}, we realise that it only uses the existence of certain primitive recursive functions plus the right induction rule. Thus, on the semantic side we also have the following result:

\begin{corollary}[PNO also satisfies the left induction rule]\label{cor:leftindpno}
  Let $N$ be a PNO in $\mc C$. For any subobject $\varphi$ of $X$ and $\psi$ of $X \times N$ in $\mc C$, if we have
  \[ X \models \psi(x,\ms 0) \vdash \varphi(x), \quad X \times N \models \psi(x,\ms sn) \vdash \varphi(x) \vee \psi(x,n), \]
  then we also have
  \[ X \times N \models \psi(x,n) \vdash \varphi(x). \]
\end{corollary}
The left induction rule will be used quite a lot when we construct bounded universal quantifiers in $\mc C$ in the next two subsections.

\subsection{Bounded $\mu$-operator}
Though the internal logic of a coherent category $\mc C$ does not support the interpretation of universal quantification in general, we can show it supports the construction of bounded universal quantification. One intermediate step is to construct \emph{bounded $\mu$-operators} for certain subobjects in $\mc C$.

We say a subobject $\varphi$ of $X$ is \emph{complemented} if there exists another subobject $\qsi\varphi$ of $X$, such that
\[ X \models \varphi \wedge \qsi\varphi \vdash \bot, \quad X \models \top \vdash \varphi \vee \qsi\varphi. \]

\begin{example}
  Over a PNO $N$ in $\mc C$, $\ms 0$ and $\ms s$ defines a pair of complemented subobjects of $N$, because we have
  \[ N \models x = \ms 0 \vee \exists y(x = \ms s y), \quad N \models \ms sx = \ms 0 \vdash \bot. \]
  The equality $=$ and the order relations $<,\le$ on $N$ in $\mc C$ are also complemented subobjects of $N \times N$. By soundness, this follows from the trichotomy of the order we have shown to hold in $\T$ in \Cref{exm:order}. Also see \citet[A2.5]{johnstone2002sketches} for a semantic proof.
\end{example}

\begin{lemma}\label{lem:compchar}
  If $\varphi$ is a complemented subobject of $X$, then there is a map $\ms c_{\varphi} : X \to N$, such that the following are pullback squares,
  \[
  \begin{tikzcd}
    \varphi \ar[d, hook] \ar[r] & 1 \ar[d, "{\ms 0}"] \\
    X \ar[r, "{\ms c_{\varphi}}"'] & N
    \arrow["\lrcorner"{anchor=center, pos=0.125}, draw=none, from=1-1, to=2-2]
  \end{tikzcd} \quad\quad
  \begin{tikzcd}
    \qsi\varphi \ar[d, hook] \ar[r] & 1 \ar[d, "{\ms 1}"] \\
    X \ar[r, "{\ms c_{\varphi}}"'] & N
    \arrow["\lrcorner"{anchor=center, pos=0.125}, draw=none, from=1-1, to=2-2]
  \end{tikzcd}
  \]
\end{lemma}
\begin{proof}
  Since $\varphi$ is complemented, $X$ can be decomposed into a coproduct,
  \[ \varphi \sqcup \qsi\varphi \cong X, \]
  where $\qsi\varphi$ is its complement. Now let $\ms c_{\varphi}$ be the universally induced map from this coproduct,
  \[ \ms c_{\varphi} = [\ms 0,\ms 1] : \varphi \sqcup \qsi\varphi \cong X \to N. \]
  This gives the required map.
\end{proof}

For such complemented subobjects $\varphi$, we will call $\ms c_{\varphi}$ its \emph{character}. The character can be used to define bounded $\mu$-operator on complemented subobjects:

\begin{definition}[Bounded $\mu$-operator]\label{def:mu}
  For any complemented subobject $\varphi(x,y)$ of $X \times N$, we define a function $\mu_{\varphi}(x,y) : X \times N \to N$ by internal primitive recursion as follows,
  \[ \mu_{\varphi}(x,y) =
    \begin{cases}
      0 & y = 0 \\
      \mu_{\varphi}(x,z) + \ms c_{\varphi}(x,\mu_{\varphi}(x,z)) & y = z+1
    \end{cases}
  \]
  where $\ms c_{\varphi}$ is the character of $\varphi$.
\end{definition}
\noindent
The existence of $\ms c_{\varphi}$ for complemented subobjects by \Cref{lem:compchar}, and the fact that primitive recursion can be defined internally in $\mc C$ by \Cref{lem:pnoprimc}, shows that $\mu_{\varphi}$ is well-defined. We prove that this function indeed defines the bounded $\mu$-operator $\muq z y\varphi(x,y)$ internally in $\mc C$:

\begin{lemma}\label{lem:compmu}
  If $\varphi(x,y)$ is a complemented subobject on $X \times N$, then we have:
  \begin{itemize}
  \item $X \times N \models \mu_{\varphi}(x,y) \le y$;
  \item $X \times N^2 \models z < \mu_{\varphi}(x,y) \vdash \qsi\varphi(x,z)$;
  \item $X \times N \models \qsi\varphi(x,\mu_{\varphi}(x,y)) \vdash y = \mu_{\varphi}(x,y)$. 
  \item $X \times N \models \varphi(x,y) \vdash \varphi(x,\mu_{\varphi}(x,\ms sy))$.
  \end{itemize}
\end{lemma}
\begin{proof}
  The first can be easily proven inductively, since $\ms c_{\varphi}$ has value either 0 or 1. We prove the second property by the \emph{left} induction rule on $y$:
  \begin{itemize}
  \item Base case: Since $\mu_{\varphi}(x,y) = 0$, $z < \mu_{\varphi}(x,y)$ will be equivalent to $\bot$, hence $z < \mu_{\varphi}(x,y) \vdash \qsi\varphi(x,y)$ trivially holds.
  \item Inductive case: We need to show $z < \mu_{\varphi}(x,\ms sy) \vdash \qsi\varphi(x,z) \vee z < \mu_{\varphi}(x,y)$. We distinguish two cases. If $\varphi(x,\mu_{\varphi}(x,y))$ holds, then by definition $\mu_{\varphi}(x,\ms sy) = \mu_{\varphi}(x,y)$, hence $z < \mu_{\varphi}(x,\ms sy) \vdash z < \mu_{\varphi}(x,y)$ holds. If $\qsi\varphi(x,\mu_{\varphi}(x,y))$ holds, then $\mu_{\varphi}(x,\ms sy) = \ms s \mu_{\varphi}(x,y)$. $z < \mu_{\varphi}(x,\ms sy)$ can further be distinguished by two cases: $z < \mu_{\varphi}(x,y)$, and we are done again; $z = \mu_{\varphi}(x,y)$, then by assumption $\qsi\varphi(x,z)$ holds.
  \end{itemize}
  For the third property, we prove the equivalent sequent by left induction,
  \[ \qsi\varphi(x,\mu_{\varphi}(x,y)) \wedge \mu_{\varphi}(x,y) < y \vdash \bot. \]
  This comes from the fact that $\mu_{\varphi}(x,y) \le y$ holds, hence $\mu_{\varphi}(x,y) < y$ is the complement of $y = \mu_{\varphi}(x,y)$.
  \begin{itemize}
  \item Base case: This is trivial, since $\mu_{\varphi}(x,y) < 0$ is equivalent to $\bot$.
  \item Inductive case: We need to prove the following sequent holds,
    \[ \qsi\varphi(x,\mu_{\varphi}(x,\ms sy)) \wedge \mu_{\varphi}(x,\ms sy) < \ms sy \vdash \qsi\varphi(x,\mu_{\varphi}(x,y)) \wedge \mu_{\varphi}(x,y) < y. \]
    We distinguish two cases. If $\ms c_{\varphi}(x,\mu_{\varphi}(x,y)) = 0$, i.e. $\varphi(x,\mu_{\varphi}(x,y))$ holds, then $\mu_{\varphi}(x,\ms sy) = \mu_{\varphi}(x,y)$. This way, from $\qsi\varphi(x,\mu_{\varphi}(x,\ms sy))$ we also know $\qsi\varphi(x,\mu_{\varphi}(x,y))$, leading to contradiction. On the other hand, suppose $\qsi\varphi(x,\mu_{\varphi}(x,y))$ holds, then $\mu_{\varphi}(x,\ms sy) = \ms s\mu_{\varphi}(x,y)$, thus from $\mu_{\varphi}(x,\ms sy) < \ms sy$ we would get $\mu_{\varphi}(x,y) < y$. Hence, the consequence also holds. 
  \end{itemize}
  The final property is implied by the previous properties. Since $\varphi$ is complemented, either $\varphi(x,\mu_{\varphi}(x,\ms sy))$, or $\qsi\varphi(x,\mu_{\varphi}(x,\ms sy))$. In the latter case, by the third property we have $\mu_{\varphi}(x,\ms sy) = \ms sy$. Now since $y < \ms sy$, the second property implies $\qsi\varphi(x,y)$, contradicting the assumption $\varphi(x,y)$. Hence, we must have $\varphi(x,\mu_{\varphi}(x,\ms sy))$.
\end{proof}

\subsection{Bounded universal quantifier}
Using bounded $\mu$-operator, we can proceed to define \emph{bounded universal quantification} for complemented subobjects of $X \times N$ for any object $X$:

\begin{definition}[Bounded universal quantification for complemented subobjects]\label{def:bduniqt}
  If $\varphi(x,y)$ is complemented over $X \times N$, then we define the bounded universal quantification on $\varphi$ as follows,
  \[ \buq zy\varphi(x,z) := y = \mu_{\qsi\varphi}(x,y). \]
\end{definition}

From the above definition, if $\varphi$ is complemented, so is $\buq zy\varphi(x,z)$, and hence similarly for $\beq zy\varphi(x,z)$. Using the properties we have established for the bounded $\mu$-operator in \Cref{lem:compmu}, we can also show that the above definition of bounded quantification is internally correct:

\begin{lemma}\label{lem:compbduniqt}
  If $\varphi(x,y)$ is complemented over $X \times N$, then $\buq zy\varphi(x,z)$ satisfies the following universal property: For any subobject $\psi(x,y)$ on $X \times N$, 
  \begin{align*}
    X \times N &\models \psi(x,y) \vdash \buq zy\varphi(x,z) \\
    \eff X \times N^2 &\models \psi(x,y) \wedge z < y \vdash \varphi(x,z)
  \end{align*}
\end{lemma}
\begin{proof}
  Suppose $X \times N \models \psi(x,y) \vdash y = \mu_{\qsi\varphi}(x,y)$. By \Cref{lem:compmu} we know that $z < \mu_{\qsi\varphi}(x,y) \vdash \varphi(x,z)$ holds, hence $X \times N^2 \models \psi(x,y) \wedge z < y \vdash \varphi(x,z)$. On the other hand, suppose $\psi(x,y) \wedge z < y \vdash \varphi(x,z)$ holds. We distinguish two cases. If $y = \mu_{\qsi\varphi}(x,y)$, this is exactly what we want. If $\mu_{\qsi\varphi}(x,y) < y$, then by assumption we have
  \[ X \times N \models \psi(x,y) \wedge \mu_{\qsi\varphi}(x,y) < y \vdash \varphi(x,\mu_{\qsi\varphi}(x,y)). \]
  By \Cref{lem:compmu} again, $\varphi(x,\mu_{\qsi\varphi}(x,y)) \vdash y = \mu_{\qsi\varphi}(x,y)$ also holds. Either way, we have the desired result.
\end{proof}

However, to fully interpret the bounded universal quantification in $\T$, we also need to construct those for $\Sigma_1$-objects as well. We first define what a $\Sigma_1$-subobject is:

\begin{definition}[$\Sigma_1$-subobject]
  A subobject $\psi$ of $X$ is $\Sigma_1$, if there exists a complemented subobject $\varphi$ of $X \times N^{n}$ for some $n$, such that
  \[ X \models \psi(x) \dashv\vdash \exists\ov y\varphi(x,\ov y). \]
\end{definition}

With bounded $\mu$-operator and bounded universal quantifier defined for complemented subobjects, we can show there is a certain \emph{choice principle} for $\Sigma_1$-subobjects of $X \times N$ in $\mc C$. 

\begin{lemma}[$\Sigma_1$-minimisation]\label{lem:sigmin}
  For any $\Sigma_1$-subobject $\exists y\varphi(x,y)$ of $X$ with $\varphi(x,y)$ complemented over $X \times N$, the following holds,
  \[ X \models \exists y\varphi(x,y) \vdash \exists y(\varphi(x,y) \wedge \buq zy\qsi\varphi(x,z)). \]
\end{lemma}
\begin{proof}
  It is equivalent to show that
  \[ X \times N \models \varphi(x,u) \vdash \exists y(\varphi(x,y) \wedge \buq zy\qsi\varphi(x,z)). \]
  By definition $\varphi$ is complemented, hence we have a function $\mu_{\varphi}$ computing its bounded minimisation. From \Cref{lem:compmu}, we know
  \[ X \times N \models \varphi(x,u) \vdash \varphi(x,\mu_{\varphi}(x,\ms su)) \wedge \buq z{\mu_{\varphi}(x,\ms su)}\qsi\varphi(x,z). \]
  This completes the proof.
\end{proof}

\begin{corollary}[$\Sigma_1$-choice]\label{cor:sigmachoice}
  If a $\Sigma_1$-subobject $\exists y\varphi(x,y)$ of $X$ is valid in $\mc C$, viz. $X \models \exists y\varphi(x,y)$, then there exists a function $f : X \to N$ such that
  \[ X \models \varphi(x,f(x)). \]
\end{corollary}
\begin{proof}
  \Cref{lem:sigmin} suggests that we have
  \[ X \models \exists y(\varphi(x,y) \wedge \buq zy\qsi\varphi(x,z)). \]
  This implies that $\varphi(x,y) \wedge \buq zy\qsi\varphi(x,z)$ is a \emph{functional} relation on $X \times N$. i.e. the value $y$ exists for any $x$, and this existence is \emph{unique}. In any coherent category, there is a one-to-one correspondence between functional relations and morphisms between objects. This is called \emph{functional completeness}; see~\citet[D1.3.12]{johnstone2002sketches}. Hence there exists a $f : X \to N$ that computes the above uniquely determined $y$.
\end{proof}

Recall that for $\varphi(x,n)$ a subobject of $X \times N$, $X \models \exists y\varphi(x,y)$ iff the composite map from $\varphi$ to $X$ is a regular epimorphism,
\[
\begin{tikzcd}
  \varphi \ar[d, hook] \ar[dr, two heads] \\
  X \times N \ar[r, "{\pi_X}"'] & X \ar[ul, bend right, dashed, "f"']
\end{tikzcd}
\]
Now if $\varphi$ is complemented, \Cref{cor:sigmachoice} implies that there exists a \emph{section} of this epimorphism $\varphi \surj X$. This is the reason why it is called a choice principle.

Although \Cref{def:bduniqt} has only specified bounded universal quantification for complemented objects, we can in fact show that $N$ in $\mc C$ will validate a form of \emph{$\Sigma_1$-collection}, hence bounded universal quantification can be applied to any $\Sigma_1$-subobject as well:

\begin{lemma}[$\Sigma_1$-collection]\label{lem:sigcollec}
  If $\varphi$ is a complemented subobject of $X \times N^2$, then the formula $\exists z \buq{i}{y} \beq{j}{z} \varphi(x,i,j)$ will satisfy the universal property of $\buq iy \exists z\varphi(x,i,z)$, i.e. for any subobject $\psi(x,y)$ of $X \times N$, we have
  \begin{align*}
    X \times N & \models \psi(x,y) \vdash \exists z\buq iy\beq jz\varphi(x,i,j) \\
    \eff X \times N^2 & \models \psi(x,y) \wedge i < y \vdash \exists z\varphi(x,i,z).
  \end{align*}
\end{lemma}
\begin{proof}
  The left to right is trivial, and we only need to show the other direction. From \Cref{cor:sigmachoice}, there exists a section $f$ from $\psi(x,y) \wedge i < y$ to $N$, such that
  \[ X \times N^2 \models \psi(x,y) \wedge i < y \vdash \varphi(x,i,f(x,i,y)). \]
  We may then recursively define a function $g$ from $\psi(x,y) \times N$ to $N$,
  \[
    \begin{cases}
      g(x,y,\ms 0) = \ms 0 \\
      g(x,y,i+1) =
                    \begin{cases}
                      \max\set{g(x,y,i),f(x,i,y)} & i < y \\
                      g(x,y,i) & y \le i
                    \end{cases}
    \end{cases}
  \]
  $g$ is well-defined, since $f(x,i,y)$ is only applied when $\psi(x,y) \wedge i < y$ holds. Intuitively, $g$ collects the maximal value of $f(x,i,y)$, and thus we have
  \[ \psi(x,y) \wedge i < y \models f(x,i,y) \le g(x,y,i). \]
  This way, it is easy to see that
  \[ X \times N \models \psi(x,y) \vdash \buq i y\beq j{\ms sg(x,y,i)}\varphi(x,i,j), \]
  and this completes the proof.
\end{proof}

\begin{corollary}\label{cor:sigmabuq}
  We can internally construct the bounded universal quantification for any $\Sigma_1$-subobjects in a pr-coherent category.
\end{corollary}
\begin{proof}
  Direct consequence of \Cref{def:bduniqt} and \Cref{lem:sigcollec}.
\end{proof}

\begin{remark}
  A relevant comparison for the above result is the well-known fact that universal quantification over \emph{Kuratowski finite} object remains geometric; cf.~\citet{vickers1999topical,johnstone1978finiteness}.
\end{remark}

We can now fulfil the promise made in \Cref{rem:coherentT}, and briefly sketch that any formula in $\T$ will be provably equivalent, modulo $\T$, to a coherent formula. From a semantic perspective, this holds because the syntactic category of $\T$ (\Cref{def:syncatcoharith}) will be pr-coherent (\Cref{thm:ctinitial}), hence all the constructions in this section applies to $\T$. From a proof-theoretic perspective, this can be more directly observed by the fact that all the semantic development in this section only involves universal properties of PNO in a coherent category. By inspecting the proof, it is evident that the same argument can be carried out in $\T$ using the rules for coherent logic together with the induction principle. It suffice to show the following fact:

\begin{proposition}[A coherent axiomatisation of $\T$]\label{rem:tiscoherent}
  For any coherent formula $\varphi(\ov x,y)$, $\buq{y}{\ms t}\varphi(\ov x,y)$ will be equivalent over $\T$ to a coherent formula. 
\end{proposition}
\begin{proof}[Proof Sketch]
  Notice that in $\T$, every sequence of existential quantifiers $\exists \ov z\psi$ (including the empty one) can be equivalently replaced by a single existential quantifier $\exists z\psi'$, since we can encode lists of natural numbers; cf. \Cref{sec:init}. Recall that every coherent formula $\varphi(\ov x,y)$ is a finite disjunction of formulas $\exists\ov z\psi_i$, where each $\psi_i$ is a finite conjunctions of atomic formulas; cf.~\citet[D1.3.8]{johnstone2002sketches}. Hence, every coherent formula $\varphi(\ov x,y)$ in $\T$ is provably equivalent to one of the form $\exists z\psi(\ov x,y,z)$, with $\psi(\ov x,y,z)$ a finite disjunction of finite conjunctions of atomic formulas. 

  Now notice that atomic formulas in $\T$ are all of the form $\ms s = \ms t$ or $\ms s < \ms t$ for terms $\ms s,\ms t$, which are \emph{complemented}. Hence, $\psi(\ov x,y,z)$, as a finite disjunction of finite conjunctions of them, is also complemented. According to (the proof of) \Cref{lem:sigcollec}, $\buq y{\ms t}\exists z\psi(\ov x,y,z)$, is equivalently $\exists z\buq y{\ms t}\beq jz\psi(\ov x,y,j)$. By (the proof of) \Cref{lem:compbduniqt}, $\buq y{\ms t}\beq jz\psi(\ov x,y,j)$ is furthermore equivalent to an atomic formula in $\T$. This way, $\buq y{\ms t}\varphi(\ov x,y)$ can be replaced by an equivalent coherent formula over $\T$.
\end{proof}

\subsection{Stability under base change}
To summarise, we have shown in this section that for a PNO in a coherent category $\mc C$: (a) We can construct primitive recursive functions internally; (b) The induction principles (IndR) and (IndL) are valid for the PNO in $\mc C$; (c) We can compute the bounded $\mu$-operator for any complemented subobject; (d) We can construct bounded universal quantification for any $\Sigma_1$-subobjects.

To show $\T$ presents the initial pr-coherent arithmetic, we also need to show that all these constructions will be \emph{preserved} by appropriate functors. They are referred to as \emph{pr-coherent functors}:

\begin{definition}[pr-coherent functor]\label{def:prcoherent}
  Let $(\mc C,N)$ and $(\mc D,M)$ be two pr-coherent categories. A \emph{pr-coherent functor} $F : \mc C \to \mc D$ is a coherent functor that preserves the PNO.
\end{definition}

We stress that for a functor to preserve the PNO, just like asking the functor to preserve any other types of universal structures in a category, it is \emph{not} enough to simply require that $F(N) \cong M$, but it should also preserve the structures $\ms 0$ and $\ms s$,
\[
\begin{tikzcd}
  F(1) \ar[d, "{\cong}"'] \ar[r, "{F\ms 0}"] & FN \ar[d, "{\cong}"] \\
  1 \ar[r, "{\ms 0}"'] & M
\end{tikzcd} \quad\quad
\begin{tikzcd}
  F(N) \ar[d, "{\cong}"'] \ar[r, "{F\ms s}"] & FN \ar[d, "{\cong}"] \\
  M \ar[r, "{\ms s}"'] & M
\end{tikzcd}
\]
Notice that any coherent functor $F$ preserves all finite products, and in particular the terminal object. We then observe that $F$ preserving the zero term and the successor function already guarantees it preserves all the other primitive recursive constructions:

\begin{lemma}\label{lem:pr-pr}
  Let $F : (\mc C,N) \to (\mc D,M)$ be a pr-coherent functor. Then for any $f : A \to X$ and $g : X \to X$, the following diagram commutes,
  \[
  \begin{tikzcd}
    FA \times FN \cong F(A \times N) \ar[r, "{F\rec_{f,g}}"] \ar[d, "{\cong}"'] & FX \ar[d, equal] \\
    FA \times M \ar[r, "{\rec_{Ff,Fg}}"'] & FX
  \end{tikzcd}
  \]
\end{lemma}
\begin{proof}
  Apply the functor $F$ to the diagram of construction of $\rec_{f,g}$ and composing with the isomorphism between $FN$ and $M$, we obtain
  \[
  \begin{tikzcd}
    & FA \times M \ar[d, "{\cong}"] & FA \times M \ar[l, "{\id \times \ms s}"'] \ar[d, "{\cong}"] \\
    FA \ar[ur, "{\pair{\id,\ms 0}}"] \ar[dr, "{Ff}"] \ar[r, "{\pair{\id,F\ms 0}}"] & FA \times FN \ar[d, "{F\rec_{f,g}}"] & FA \times FN \ar[l, "{\id \times F\ms s}"'] \ar[d, "{F\rec_{f,g}}"] \\
    & FX & FX \ar[l, "g"]
  \end{tikzcd}
  \]
  The upper diagram commutes because $F$ preserves $\ms 0$ and $\ms s$. Now the desired property follows from the uniqueness of the function $\rec_{Ff,Fg}$.
\end{proof}

Besides the primitive recursive structure, $F$ also preserves the logical structure of a subobject being complemented or $\Sigma_1$. The former is because $F$ by definition would preserve the distributive lattice structure of the subobject lattice, and complements are uniquely determined by \emph{equational} properties: $\varphi \wedge \qsi\varphi = \bot$ and $\varphi \vee \qsi\varphi = \top$. The preservation of $\Sigma_1$-subobjects is then implied by the fact that $F$ further preserves existential quantification by definition of being coherent. The upshot is that a pr-coherent functor will preserve everything that we have introduced in this section:

\begin{proposition}[Stability under pr-coherent functors]\label{lem:stability}
  Take a pr-coherent functor $F : (\mc C,N) \to (\mc D,M)$. It preserves primitive recursion, bounded $\mu$-operator, and bounded universal quantification for $\Sigma_1$-subobjects.
\end{proposition}
\begin{proof}
  Recall the construction of bounded $\mu$-operator for complemented subobjects in \Cref{def:mu} from primitive recursion. We first note that $F$ preserves the character $\ms c_{\varphi}$ of a complemented subobject $\varphi$ of $X$, in the sense that $F\ms c_{\varphi} \cong \ms c_{F\varphi}$. This is simply because $F$ being coherent preserves the disjoint coproduct $X \cong \varphi \sqcup \qsi\varphi$, as well as $\ms 0$ and $\ms s$. Then it also preserves $\mu_{\varphi}$ by \Cref{lem:pr-pr}. Since bounded universal quantification of a complemented subobject is defined by bounded $\mu$-operator from \Cref{def:bduniqt}, it is preserved by $F$ as well. Finally, in \Cref{lem:sigcollec} we have reduced bounded universal quantification of $\Sigma_1$-subobjects to the bounded universal quantification for complemented objects, hence it is also preserved by $F$.
\end{proof}

\section{Initiality of coherent arithmetic}\label{sec:init}

The aim of this section is to show that the syntactic category $\mc C[\T]$ of our coherent arithmetic $\T$ is \emph{initial} among pr-coherent categories. The main point we need to verify is the existence of a PNO in $\mc C[\T]$. Once we have shown that, the initiality of $\mc C[\T]$ follows quite easily from our semantic analysis in \Cref{sec:pno} and how the theory $\T$ of coherent arithmetic is constructed in \Cref{def:coharith}.

We start with the definition of the syntactic category $\mc C[\T]$:

\begin{definition}[Syntactic category of coherent arithmetic]\label{def:syncatcoharith}
  The syntactic category $\mc C[\T]$ of coherent arithmetic $\T$ is a category with
  \begin{itemize}
  \item Objects: formulas with \emph{contexts} $\varphi(\ov x)$, identified up to $\alpha$-equivalence;
  \item Morphisms: $\theta : \varphi(\ov x) \to \psi(\ov y)$ is a provably functional formula $\theta(\ov x,\ov y)$ with domain $\varphi(\ov x)$ and codomain $\psi(\ov y)$, up to $\T$-provable equivalence.
  \end{itemize}
\end{definition}

Let us explain the above definition in greater detail. A formula with a context $\varphi(\ov x)$ is a formula $\varphi$, plus a list of variables $\ov x$ that contains all the free variables in $\varphi$. For instance, $\top()$, $\top$ with the empty context, would be an object in $\mc C[\T]$. $\top$ equipped with different contexts are considered as \emph{different} objects, say $\top (x)$ would be different from $\top()$. However, we only identify them up to $\alpha$-equivalence, viz. up to renaming of bound variables and substitution of free variables. This means $x = x(x)$ would be considered as the \emph{same} object as $y = y(y)$. This way, we can always assume two different object $\varphi(\ov x)$ and $\psi(\ov y)$ have \emph{disjoint} contexts, and we will assume this whenever we have chosen different variable names $x,y$ for them.

The object $\top(x)$ is special in $\mc C[\T]$, and henceforth we will denote it as $N$. Similarly, $\top(\ov x)$ for a list of variables of length $n$ will be denoted as $N^n$. They are indeed the $n$-fold product of $N$ in $\mc C[\T]$. When $n = 0$, we have the object $\top()$, which is terminal in $\mc C[\T]$. With a bit of abuse of notation, we often omit mentioning the empty context and simply write it as $\top$. As the name suggests, our main goal is to show that $N$ is a PNO in $\mc C[\T]$.

A morphism from $\varphi(\ov x)$ to $\psi(\ov y)$ is itself a formula $\theta(\ov x,\ov y)$, such that $\T$ proves its functionality with domain $\varphi(\ov x)$ and codomain $\psi(\ov y)$:
\begin{align*}
  \theta(\ov x,\ov y) &\vdash_{\ov x,\ov y} \varphi(\ov x) \wedge \psi(\ov y) \\
  \varphi(\ov x) &\vdash_{\ov x} \exists\ov y\theta(\ov x,\ov y) \\
  \theta(\ov x,\ov y) \wedge \theta(\ov x,\ov z) &\vdash_{\ov x,\ov y,\ov z} \ov y = \ov z
\end{align*}

More precisely, a morphism from $\varphi(\ov x)$ to $\psi(\ov y)$ are determined only up to $\T$-provable equivalence. This means if we have another provably functional formula $\sigma(\ov x,\ov y)$ from $\varphi(\ov x)$ to $\psi(\ov y)$, and $\T$ proves
\[ \theta(\ov x,\ov y) \dashv\vdash_{\ov x,\ov y} \sigma(\ov x,\ov y), \]
then $\theta(\ov x,\ov y)$ and $\sigma(\ov x,\ov y)$ will be considered as the \emph{same} morphism in $\mc C[\T]$.

Composition is constructed as follows: Given $\theta : \varphi(\ov x) \to \psi(\ov y)$ and $\sigma : \psi(\ov y) \to \chi(\ov z)$, the composition of $\theta$ and $\sigma$ is defined to be the following morphism,
\[ \sigma \circ \theta := \exists \ov y(\theta(\ov x,\ov y) \wedge \sigma(\ov y,\ov z)). \]
One can show that this is well-defined, and $\mc C[\T]$ forms a category under this construction. As an easy example, any term $t(\ov x)$ will induce a morphism from $N^n$ to $N$, because the formula $t(\ov x) = y$ is evidently provably functional. Composition of these morphisms are equivalent to substitution of terms. 

The above construction of syntactic category actually works for arbitrary coherent theories, not just the coherent arithmetic $\T$. And through this construction, categorical logic identifies a close relationship between coherent categories and coherent logic. The following result is well-known:

\begin{proposition}\label{prop:cohsyn}
  The syntactic category of any coherent theory is a coherent category. In particular, $\mc C[\T]$ is a coherent category.
\end{proposition}
\begin{proof}
  The result for coherent theories is well-known; cf.~\citet[D1.4]{johnstone2002sketches}. Though as mentioned in \Cref{rem:coherentT} our construction of $\T$ strictly speaking does not make it a coherent category, the proof in \emph{loc. cit.} relies only on the fact that a theory has the logical connectives in coherent logic and they satisfy the usual provability conditions. For instance, it suffices for the theory $\T$ to have conjunction and equality for $\mc C[\T]$ to be finitely complete; it suffices for $\T$ to further have existential quantifier for $\mc C[\T]$ to be regular; and finally, it suffices for $\T$ to have disjunction for $\mc C[\T]$ to be coherent. This is in particular true for $\T$.
\end{proof}

Thus, to show $\mc C[\T]$ is a pr-coherent category, we only need to prove that $N$ is a PNO in $\mc C[\T]$. As mentioned previously, the terms $\ms 0$ and $\ms s(x)$ already gives us morphisms $\ms 0 : \top \to N$ and $\ms s : N \to N$. We then need to show that $N$ equipped with these morphisms satisfies the universal property of PNO, and by definition we need to show that we can perform primitive recursion for morphisms internally in $\T$. 

As usual, this needs encoding of finite lists of numbers in $\T$. The exact detail of the coding does not matter. Here, we use the usual G\"odel coding of numbers through the $\beta$-function. We choose a bijective coding of pairs, 
\[ \#\pair{n,m} = \frac12(n+m)(n+m+1) + n. \]
This pairing function and its two inverses are evidently primitive recursive, hence there will be terms $\pri,\fst,\snd$ in $\T$ corresponding to these functions, and $\T$ can prove the following result,
\[ \top \vdash_{x,y} \fst(\pri(x,y)) = x, \quad \top \vdash_{x,y} \snd(\pri(x,y)) = y. \]
These imply that the following is derivable,
\[ \pri(x,y) = \pri(z,w) \vdash_{x,y,z,w} x = z \wedge y = w. \]
The $\beta$-function can then be defined as follows,
\[ \beta(x,i) :\equiv \rem(\fst(x),\ms s(\snd(x)\cdot\ms si)), \]
where $\rem$ is the remainder function. As usual, we will record the length of the sequence at the first entry, viz. we define
\[ \lh(x) :\equiv \beta(x,\ms 0). \]
We will also abbreviate the $i$-th entry of a sequence $(x)_i$ as below,
\[ (x)_i :\equiv \beta(x,\ms si). \]
We use these functions to show the following theorem:

\begin{proposition}\label{prop:ctprcoh}
  $N$ is a PNO in $\mc C[\T]$.
\end{proposition}
\begin{proof}[Proof Sketch]
  Consider maps $\gamma : \varphi \to \psi$ and $\theta : \psi \to \psi$ in $\mc C[\T]$. For simplicity, we assume the contexts of $\varphi,\psi$ are single variables. To show $N$ is a PNO in $\mc C$, we need to construct a unique map $\rec_{\gamma,\theta} : \varphi \times N \to \psi$, and we define it to be the following formula,
  \begin{align*}
    \rec_{\gamma,\theta}(x,n,y) :\equiv \exists l(\lh(l) = \ms sn \wedge \gamma(x,(l)_{\ms 0}) \wedge \buq un\theta((l)_u,(l)_{\ms su}) \wedge (l)_n = y).
  \end{align*}
  We indicate what we need to show. Firstly, we need to prove $\rec_{\gamma,\theta}$ is a well-defined morphism. This means it should respect the domain and codomain,
  \[ \rec_{\gamma,\theta}(x,n,y) \vdash_{x,n,y} \varphi(x) \wedge \psi(y), \]
  and it should have a provably unique value,
  \[ \varphi(x) \vdash_{x,n} \exists y\rec_{\gamma,\theta}(x,n,y), \quad \rec_{\gamma,\theta}(x,n,y) \wedge \rec_{\gamma,\theta}(x,n,z) \vdash_{x,n,y,z} y = z. \]
  Respecting the domain and codomain and the uniqueness of value can simply be proven by a case distinction on $n$, and using the corresponding properties of $\gamma$ and $\theta$. The existence of value can be proven by induction on $n$. When verifying these facts, we will need certain basic operations on lists like concatenation, and finding the code of a list consisting of a single number. But these operations are primitive recursive, hence exist in $\T$.
  
  The definition of $\rec_{\gamma,\theta}$ should make it clear that the following diagram commutes,
  \[
  \begin{tikzcd}
    \varphi(x) \ar[dr, "{\gamma}"'] \ar[r, "{\pair{\id,\ms 0}}"] & \varphi(x) \times N \ar[d, "{\rec_{\gamma,\theta}}"] & \varphi(x) \times N \ar[l, "{\id \times \ms s}"'] \ar[d, "{\rec_{\gamma,\theta}}"] \\
    & \psi(y) & \psi(y) \ar[l, "{\theta}"]
  \end{tikzcd}
  \]
  and we need to verify the uniqueness of the morphism $\rec_{\gamma,\theta}$, in the sense that for any other morphism $\sigma$ from $\varphi \times N$ to $\psi$ that makes the above diagram commute, $\sigma$ and $\rec_{\gamma,\theta}$ will be provably equivalent in $\T$. There should be no problem for the reader familiar with basic proof theory in arithmetic to realise the validity of the previously claimed fact, as they are well-known consequences of $\Sigma_1$-induction. We have recorded the full proof in \ref{app}.
\end{proof}

As an example, we can verify that our function symbols in \emph{PrimRec} in $\T$ are indeed computing internal primitive recursions in the following sense:

\begin{example}
  Given an $n$-ary $\ms g$ and $n\!+\!2$-ary $\ms h$ in \emph{PrimRec}, let $\ms f$ be $\mathrm{Pr}[\ms g,\ms h]$. By assumption, $\T$ contains the following axioms,
  \[ \top \vdash_{\ov x} \ms f(\ov x,\ms 0) = \ms g(\ov x), \quad \top \vdash_{\ov x,y} f(\ov x,\ms sy) = \ms h(\ov x,y,\ms f(\ov x,y)). \]
  These axioms makes it clear that the following diagram commutes,
  \[
  \begin{tikzcd}
    N^n \ar[r, "{\pair{\id,\ms 0}}"] \ar[dr, "{\pair{\id,\ms 0,\ms g}}"'] & N^{n+1} \ar[d, "{\pair{\id,\ms f}}"] & N^{n+1} \ar[l, "{\id \times \ms s}"'] \ar[d, "{\pair{\id,\ms f}}"] \\
    & N^{n+2} & N^{n+2} \ar[l, "{\pair{\pi,\ms h}}"]
  \end{tikzcd}
  \]
  where $\pi : N^{n+2} \to N^{n+1}$ projects the first $n+1$ entries. By uniqueness of $\rec$, $\ms f$ will be the same function that constructed out of the PNO structure of $N$ as specified by \Cref{lem:pnoprimc}.
\end{example}

As mentioned before, Proposition~\Ref{prop:ctprcoh} almost immediately implies one of the main results of this paper:

\begin{theorem}\label{thm:ctinitial}
  $\mc C[\T]$ is the \emph{initial} pr-coherent category.
\end{theorem}
\begin{proof}
  Proposition~\Ref{prop:ctprcoh} implies that $\mc C[\T]$ is first of all a pr-coherent category. Combining \Cref{lem:pnoprimc}, \Cref{lem:indpno}, and \Cref{cor:sigmabuq}, it follows that any PNO $M$ in a coherent category $\mc D$ would consist of a \emph{model} of $\mc C[\T]$, hence providing a functor $F_M : \mc C[\T] \to \mc D$ mapping $N$ to $M$; cf.~\citet[D1.2]{johnstone2002sketches}. Now \Cref{lem:stability} ensures this functor is indeed coherent, hence $F_M$ is pr-coherent. For any other pr-coherent functor $F$, it must preserve the PNO structure. Since all the objects $\varphi(\ov x)$ in $\T$ are generated by primitive recursion and coherent logic, the value of $F$ is completely determined up to isomorphism, and we must have $F(\varphi(\ov x)) \cong F_M(\varphi(\ov x))$ for any $\varphi(\ov x)$ in $\mc C[\T]$, thus $F \cong F_M$. This proves the initiality of $\mc C[\T]$.
\end{proof}

Our next goal is to characterise the definable functions in the initial pr-coherent category, i.e. morphisms between powers of $N$ in $\mc C[\T]$. From \Cref{def:syncatcoharith}, they corresponds exactly to provably total functions in $\T$. In fact, they are all strongly $\Sigma_1$-representable, because all formulas in $\T$ are $\Sigma_1$ by construction.

From \Cref{thm:ctinitial}, we already know that this class of functions contain all primitive recursive functions, because they already appear as terms in $\T$. As mentioned in \Cref{sec:intro}, this class of functions turns out to be exactly primitive recursive functions, and the next section is devoted to the proof of this fact.

\section{Definable functions in the initial pr-coherent category}\label{sec:defin}

As mentioned in the introduction, our strategy is to define a category $\PriM$ of primitive recursive functions between recursively enumerable sets, and show it is a pr-coherent category. Recall that a subset $S$ of $\N^k$ is recursively enumerable if there exists a computable function whose range is $S$.\footnote{We view the nowhere defined function to be computable. This gives us the empty set $\emptyset$ as a recursively enumerable set.} The reason we are particularly interested in these class of sets is due to the following fundamental result; cf.~\citet[Prelimaries]{pudlak2017}:

\begin{proposition}
  A subset $S$ of $\N^k$ is recursively enumerable iff it can be defined by a $\Sigma_1$-formula, i.e. a formula in $\T$.
\end{proposition}

Thus in particular, the functor $\mc C[\T] \to \Set$ induced by the standard natural numbers $\N$ in $\Set$ will map any object in $\mc C[\T]$ to a recursively enumerable set.

To define a notion of primitive recursive functions between recursively enumerable sets, the following fact is crucial: Using Kleene's $T$-predicate, we can show that every recursively enumerable set, if non-empty, has a primitive recursive enumeration. More precisely, if $S \subseteq \N^k$ is recursively enumerable, then there exists a family $\set{s_i}_{1 \le i\le k}$ of primitive recursive functions, such that $S$ is the range of the function
\[ \pair{s_i}_{1 \le i\le k} : \N \to \N^k. \]
This function will be called a \emph{primitive recursive enumeration} of $S$, and this fact appears e.g. in~\citet{rosser1936extensions}. Starting from this, we may define the category $\PriM$ as follows:

\begin{definition}[Category of primitive recursive functions between recursively enumerable sets]\label{def:prim}
  The category $\PriM$ is defined as follows:
  \begin{itemize}
  \item Objects are either $\emptyset$, or pairs $(S,s)$, where $S$ is a non-empty recursively enumerable set, and $s : \N \to S$ is a chosen primitive recursive enumeration of $S$. We will write $s(n)$ also as $s_n$, to indicate we treat $n$ as an index for the element $s_n \in S$.
  \item Morphisms out of $\emptyset$ are unique, and no object maps into $\emptyset$ except for itself. A morphism from $(S,s)$ to $(T,t)$ is a function $f : S \to T$, such that there exists some \emph{primitive recursive} function $\qsi f$ making the following diagram commute,
    \[
    \begin{tikzcd}
      \N \ar[r, "s"] \ar[d, "{\qsi f}"'] & S \ar[d, "f"] \\
      \N \ar[r, "t"'] & T
    \end{tikzcd}
    \]
    In other words, morphisms in $\PriM$ are functions that are \emph{tracked} by some primitive recursive function on the level of codes. Composition of morphism $f,g$ will be the usual composition of functions $f \circ g$. This is well-defined because $f \circ g$ is evidently tracked by $\qsi f \circ \qsi g$.
  \end{itemize}
\end{definition}

According to the above definition, for any $f : (S,s) \to (T,t)$ in $\PriM$, the composition $f \circ s$ will be primitive recursive, because it is equal to $t \circ \qsi f$, which is primitive recursive. In particular, in $\PriM$, once we have chosen some bijective primitive recursive coding $x : \N \to \N^n$ with a primitive recursive inverse, morphisms from $(\N^n,x)$ to $(\N,\id)$ will be exactly the primitive recursive functions.\footnote{As we will see in \Cref{lem:primlex}, $(\N^n,x)$ will indeed be the $n$-fold product of $(\N,\id)$ in $\PriM$.}

There is an evident forgetful functor
\[ U : \PriM \to \Set. \]
It sends $\emptyset$ to the empty set, and sends $(S,s)$ to $S$. For morphisms, it sends the morphisms out of $\emptyset$ as the unique map out of the empty set, and for $f : (S,s) \to (T,t)$, $U$ forgets about the information that $f$ can be tracked by some primitive recursive function $\qsi f$ and sends it to the set-theoretic function $f : S \to T$.

Our goal is to show that $\PriM$ is a pr-coherent category, and this forgetful functor is a pr-coherent functor. One immediate observation is that $U$ is \emph{faithful}, hence it reflects monomorphisms. This means that if $f$ is \emph{injective}, then it is also a monomorphism in $\PriM$. 

\begin{lemma}\label{lem:primlex}
  $\PriM$ has finite limits, and $U$ preserves them.
\end{lemma}
\begin{proof}
  Evidently $\set 0$ with its unique enumeration is a terminal object in $\PriM$, and $U$ preserves it. For pullbacks, it suffices to consider the case where all objects involved are not $\emptyset$. Suppose we have the following morphisms in $\PriM$,
  \[ f : (S,s) \to (X,x), \quad g : (T,t) \to (X,x). \]
  We first look at the set-theoretic pullback as follows,
  \[
  \begin{tikzcd}
    S \times_X T \ar[d, "{\pi_0}"'] \ar[r, "{\pi_1}"] & T \ar[d, "g"] \\
    S \ar[r, "f"'] & X
  \end{tikzcd}
  \]
  If $S \times_X T$ is empty, then it is evident that $\emptyset$ is also the pullback in $\PriM$, and $U$ preserves it. If $S \times_X T$ is non-empty, then we construct an enumeration of it. Fix some pair $(i,j)$ that $f(s_i) = g(t_j)$, and fix some bijective, monotone, primitive recursive pairing function $\pri$ with inverse $\fst,\snd$,\footnote{Here we have slightly abused the notation. Previously in \Cref{sec:init}, $\pri$, $\fst$ and $\snd$ are used as function names in $\T$, while here they denote some actual functions. However, we believe this will not cause any serious confusion.} define the enumeration $y$ of $S \times_X T$ as follows,
  \[ y(n) =
    \begin{cases}
      (s_{\fst(n)},t_{\snd(n)}) & f(s_{\fst(n)}) = g(t_{\snd(n)}) \\
      (s_i,t_j) & \other
    \end{cases}
  \]
  This enumeration is evidently primitive recursive, because by definition both $f \circ s$ and $g \circ t$ are primitive recursive. Intuitively, $y$ enumerates all pairs in $S \times T$, and compare their values under $f,g$: If they are equal, then $y$ maps $n$ to this pair; if they are not, then $y$ maps $n$ to some pair chosen in advance. Notice that the two projection maps $\pi_0$, $\pi_1$ are tracked by the following functions,
  \begin{align*}
    &\qsi{\pi_0}(n) =
    \begin{cases}
      \fst(n) & f(s_{\fst(n)}) = g(t_{\snd(n)}) \\
      i & \other
    \end{cases} \\
    &\qsi{\pi_1}(n) =
    \begin{cases}
      \snd(n) & f(s_{\fst(n)}) = g(t_{\snd(n)}) \\
      j & \other
    \end{cases}
  \end{align*}
  Both of them are primitive recursive. We verify the universal property. Given $u : (Z,z) \to (S,s)$ and $v : (Z,z) \to (T,t)$ that $f \circ u = g \circ v$, there is already a uniquely determined map $\pair{u,v} : Z \to S \times_X T$. We only need to show it is tracked by some primitive recursive function $\qsi{\pair{u,v}}$,
  \[ \qsi{\pair{u,v}}(n) = \pri(\qsi u(n),\qsi v(n)). \]
  By definition it is primitive recursive, and for any $n\in\N$ we must have
  \begin{align*}
    y(\qsi{\pair{u,v}}(n)) &= (s(\fst(\qsi{\pair{u,v}}(n))),t(\snd(\qsi{\pair{u,v}}(n)))) \\
    &= (s_{\qsi u(n)},t_{\qsi v(n)}) = (u(z_n),v(z_n)) = \pair{u,v}(z_n).
  \end{align*}
  The first equality is due to the fact that $\qsi{\pair{u,v}}(n)$ codes the pair $\qsi u(n),\qsi v(n)$, and we have
  \[ f(s_{\qsi u(n)}) = f(u(z_n)) = g(v(z_n)) = g(t_{\qsi v(n)}), \]
  hence the first clause of $y$ is effective. This shows that $\pair{u,v}$ is indeed tracked by some primitive recursive function, hence $(S\times_XT,y)$ is indeed the pullback in $\PriM$, and $U$ preserves pullbacks. 
\end{proof}

\begin{remark}\label{rem:mono}
  However, it is worth noticing that, unlike the case in $\Set$, \emph{not all} monomorphisms in $\PriM$ are isomorphic to one of the form $(S,s) \hook (T,t)$, where $S \subseteq T$ and the map is given by inclusion. This is due to the fact that there exists a bijective primitive recursive function, whose inverse is \emph{not} primitive recursive. Let $f : \N \to \N$ be such a function, then it is injective, thus consists of a monomorphism in $\PriM$,
  \[ f : (\N,\id) \inj (\N,\id). \]
  It is evidently tracked by $f$ itself. It is indeed a \emph{proper} monomorphism, i.e. not isomorphic to the identity on $(\N,\id)$, because $f \cong \id$ in the subobject lattice of $(\N,\id)$ iff $f\inv$ is also primitive recursive.
\end{remark}

\Cref{rem:mono} implies that we should be more careful when defining image factorisation in $\PriM$. Given a morphism $f : (S,s) \to (T,t)$, we can factorise it as follows,
\[
\begin{tikzcd}
  (S,s) \ar[dr, two heads] \ar[rrr, "f"] & & & (T,t) \\
  & (S/f,x) \ar[r, "f"'] & (f(S),y) \ar[ur, hook]
\end{tikzcd}
\]
where $S/f$ is the \emph{quotient} of $S$ under the equivalence relation generated by $f$, and $f(S)$ is the image of $S$ under $f$. Although when restricting $f$ to a map from $S/f$ to $f(S)$ it is bijective, it may \emph{not} always be an isomorphism in $\PriM$. Evidently, the object $(S/f,x)$ is the correct image of $f$ in $\PriM$:

\begin{lemma}\label{lem:primreg}
  $\PriM$ is regular, and $U$ preserves the image factorisation.
\end{lemma}
\begin{proof}
  Again it suffices to construct the image factorisation for non-empty enumerations because the image of any map out of $\emptyset$ is given by $\emptyset$ itself. Suppose we have a morphism $f : (S,s) \to (T,t)$, we first construct a function $g : \N \to \N$ as follows,
  \[ g(n) = \muq in.f(s_i) = f(s_n). \]
  $g$ is evidently primitive recursive, because again $f \circ s$ is. Intuitively, for any $n\in\N$, $g$ computes the minimal index $i$ such that $s_i$ and $s_n$ has the same value under $f$. Let $S/f$ be the recursively enumerable set defined by the enumeration $s \circ g$. There is an evident quotient map
  \[ q : (S,s) \surj (S/f,s \circ g), \]
  sending each $x \in S$ to some $s_i$, where $i$ is the minimal index such that $s_i$ has the same value as $x$ under $f$. This quotient map is in fact tracked by the identity function,
  \[
  \begin{tikzcd}
    \N \ar[r, "s"] \ar[d, equal] & S \ar[d, two heads, "q"] \\
    \N \ar[r, "{s\circ g}"'] & S/f
  \end{tikzcd}
  \]
  because by definition, for any $n\in\N$,
  \[ q(s_n) = s(\muq in.f(s_i) = f(s_n)) = s_{g(n)}. \]
  Hence, $q$ is a well-defined morphism in $\PriM$. $f$ also induces a map
  \[ f : (S/f,s\circ g) \to (T,t). \]
  It is also tracked by $\qsi f$ in this case, because for any $n\in\N$,
  \[ f(s_{g(n)}) = f(s_n) = t_{\qsi f(n)}. \]
  This way, we have indeed obtained a factorisation as follows,
  \[
  \begin{tikzcd}
    (S,s) \ar[rr, "f"] \ar[dr, two heads, "q"'] & & (T,t) \\
    & (S/f,s \circ g) \ar[ur, hook]_f &
  \end{tikzcd}
  \]
  Furthermore, $q$ \emph{splits} in $\PriM$, because the inclusion $(S/f,s\circ g) \hook (S,s)$ is also a morphism in $\PriM$ (it is tracked by $g$), and it is a section of $q$. This shows that every morphism in $f$ actually factors as a \emph{split epi} followed by a mono, and it implies $(S/f,s\circ g)$ is the image factorisation of $f$ in $\PriM$, and the factorisation is preserved under pullback. $U$ evidently preserves image factorisation.
\end{proof}

\begin{lemma}\label{lem:primcoh}
  $\PriM$ is coherent, and $U$ preserves the coherent structure.
\end{lemma}
\begin{proof}
  We need to verify that the subobject lattice of any object $(S,s)$ will be a distributive lattice, and pullbacks preserves meets and joins in them. It suffices to show we have disjoint and universal finite coproduct in $\PriM$, since we have already shown $\PriM$ is regular. To this end, given $(S,s)$ and $(T,t)$, we construct its coproduct as $(S\sqcup T,x)$, where $S \sqcup T$ is the disjoint union of $S$ and $T$, and the enumeration is given by
  \[ x_n =
    \begin{cases}
      s_i & x = 2i \\
      t_i & x = 2i+1
    \end{cases}
  \]
  $x$ is evidently primitive recursive, and verifying its universal property is routine. $U$ clearly preserves the coproduct, and it being disjoint and universal follows from the fact that $U$ also preserves pullbacks, and coproducts in $\Set$ is disjoint and universal.
\end{proof}

\begin{remark}
  Notice that $\PriM$ is \emph{not} Boolean, nor Heyting, essentially because recursively enumerable subsets are not closed under complements. This implies that we cannot interpret all formulas in $I\Sigma_1$ in $\PriM$, but we can interpret $\T$.
\end{remark}

\begin{lemma}\label{lem:primpno}
  $(\N,\id)$ is a PNO in $\PriM$.
\end{lemma}
\begin{proof}
  For any $f : (S,s) \to (T,t)$ and $g$ from $(T,t)$ to itself, we do get a diagram of the following type,
  \[
  \begin{tikzcd}
    (S,s) \ar[r, "{\id \times 0}"] \ar[dr, "{f}"'] & (S \times \N,x) \ar[d, "{\rec_{f,g}}"] & (S\times\N,x) \ar[l, "{\id \times\ms s}"'] \ar[d, "{\rec_{f,g}}"] \\
    & (T,t) & (T,t) \ar[l, "{g}"]
  \end{tikzcd}
  \]
  Here the function $\rec_{f,g}$ is given by the following definition,
  \[
    \begin{cases}
      \rec_{f,g}(a,0) = f(a) \\
      \rec_{f,g}(a,n+1) = g(\rec_{f,g}(a,n))
    \end{cases}
  \]
  This is the same way how primitive recursion is constructed in $\Set$, thus it is already unique. We then only need to prove that $\rec_{f,g}$ is tracked by some primitive recursive function $\qsi{\rec_{f,g}}$. Recall from \Cref{lem:primlex}, the enumeration $x$ is given as follows,
  \[ x(n) = (s_{\fst(n)},\snd(n)). \]
  Hence, we may again construct $\qsi{\rec_{f,g}}$ by primitive recursion as follows,
  \[
    \begin{cases}
      \qsi{\rec_{f,g}}(n) = \qsi f(\fst(n)) & \snd(n) = 0 \\
      \qsi{\rec_{f,g}}(n) = \qsi g(\qsi{\rec_{f,g}}(\pri(\fst(n),m))) & \snd(n) = m + 1
    \end{cases}
  \]
  This makes $\qsi{\rec_{f,g}}$ primitive recursive, because if $\snd(n) = m+1$, then since we have chosen a monotone pairing function, $\pri(\fst(n),m) < n$. It is easy to verify that $\rec_{f,g}$ is tracked by $\qsi{\rec_{f,g}}$, and it implies that $U$ preserves the PNO structure. 
\end{proof}

\begin{proposition}\label{prop:primprcoherent}
  $\PriM$ is a pr-coherent category, and $U : \PriM \to \Set$ is a pr-coherent functor.
\end{proposition}
\begin{proof}
  Combine \Cref{lem:primlex},~\Ref{lem:primreg},~\Ref{lem:primcoh}, and~\Ref{lem:primpno}.
\end{proof}

The following is the main semantic result of this paper:

\begin{theorem}[Definable functions in the initial pr-coherent category]\label{thm:defninprcoh}
  Let $(\mc C,N)$ be the initial pr-coherent category. The evaluation of morphisms between powers of $N$ in $\mc C$ along the unique functor $\mc C \to \Set$ are exactly primitive recursive functions.
\end{theorem}
\begin{proof}
  By initiality of $\mc C$, we have the following diagram,
  \[
  \begin{tikzcd}
    \mc C[\T] \ar[dr] \ar[rr] & & \Set \\
    & \PriM \ar[ur]
  \end{tikzcd}
  \]
  This means the canonical interpretation of $\mc C[\T]$ into $\Set$ factors through $\PriM$, hence any such $\theta : N^n \to N$ defines a primitive recursive function from $\N^n$ to $\N$ by the definition of $\PriM$.
\end{proof}

\begin{corollary}\label{cor:ptft}
  Provably total functions in $\T$ are exactly primitive recursive functions.
\end{corollary}
\begin{proof}
  By \Cref{thm:ctinitial}, $\mc C[\T]$ is the initial pr-coherent category, and morphisms between $N$ in $\mc C[\T]$ are exactly provably total functions in $\T$.
\end{proof}

As one can see, once we have properly established the initiality of $\mc C[\T]$ and realised that $\PriM$ is a pr-coherent category, the characterisation of provably total functions in $\T$ follows as an easy consequence.

We will see more applications of the initiality of $\mc C[\T]$ in \Cref{sec:furtherp}. But before that, let us first discuss the precise relationship between the theory $\T$ of coherent arithmetic, and the theory $I\Sigma_1$. We will show in the next section that $\T$ is precisely the $\Sigma_1$-fragment of $I\Sigma_1$.

\section{Coherent arithmetic and $I\Sigma_1$}\label{sec:cohsigma}

In this section, we want to compare our coherent arithmetic $\T$ with $I\Sigma_1$. We will show that $I\Sigma_1$ is a \emph{conservative} extension of $\T$, in the sense that for any sequent $\varphi \vdash_{\ov x} \psi$ in $\T$, if it is provable in $I\Sigma_1$, then it is already provable in $\T$. This implies that the provably total functions in $\T$ coincide with the strong $\Sigma_1$-representable functions in $I\Sigma_1$. Thus, together with \Cref{cor:ptft}, the theorem mentioned in \Cref{subsec:application} follows as a consequence.

To make the above claim precise, we need an embedding of the theory $\T$ into $I\Sigma_1$. This is not automatic, because $\T$ contains additional function symbols than $I\Sigma_1$. However, this is not an essential problem: We may simply add all function symbols in \emph{PrimRec} and their corresponding axiomatisation into $I\Sigma_1$ as well. This will not change the theory $I\Sigma_1$ in any essential way, because it is well-known by the work of G\"odel that any primitive recursive function is strongly representable in $I\Sigma_1$ by $\Sigma_1$-formulas. This way, every $\T$ sequent can be viewed as an $I\Sigma_1$ sequent as well.

The conservativity result follows almost immediately by proof-theoretic analysis of $I\Sigma_1$. For this purpose, it is convenient to consider a derivation system for $I\Sigma_1$ in natural deduction style. A typical natural deduction proof will be a finite tree with possibly open leaves,
\[
  \begin{prooftree}
    \hypo{\varphi}
    \ellipsis{}{\psi}
  \end{prooftree}
\]
Besides the usual introduction and elimination rules for connectives, the natural deduction system for $I\Sigma_1$ furthermore has the following rule for induction on $\Sigma_1$-formulas $\varphi(\ov x,y)$:
\[
  \begin{prooftree}
    \hypo{}
    \ellipsis{}{\varphi(\ov x,\ms 0)}
    \hypo{[\varphi(\ov x,y)]}
    \ellipsis{}{\varphi(\ov x,\ms sy)}
    \infer 2[IND]{\varphi(\ov x,y)}
  \end{prooftree}
\]
In particular, when apply the rule IND, the open assumption $\varphi(\ov x,y)$ on the derivation of $\varphi(\ov x,\ms sy)$ can be cancelled.

\begin{proposition}\label{thm:conservative}
  $I\Sigma_1$ is conservative over $\T$.
\end{proposition}
\begin{proof}
  Let $\varphi(\ov x) \vdash_{\ov x} \psi(\ov x)$ be a sequent in $\T$, and suppose it is provable in $I\Sigma_1$. This means that there is a natural deduction proof tree in $I\Sigma_1$ having the following form,
\[
  \begin{prooftree}
    \hypo{\varphi(\ov x)}
    \ellipsis{}{\psi(\ov x)}
  \end{prooftree}
\]
By normalisation result for natural deduction of arithmetic, see e.g.~\citet{annika2015normalisation}, there is a natural deduction proof which has the \emph{subformula property}, i.e. every formula appearing in the proof tree must be subformulas of $\varphi(\ov x)$ and $\psi(\ov x)$. In particular, every formula appearing in the proof tree will be \emph{coherent}, and the rule applications are restricted to the coherent fragment of first-order logic. Then it is not hard to see the whole derivation can be carried out in $\T$ to prove $\varphi(\ov x)\vdash_{\ov x}\psi(\ov x)$ as well.
\end{proof}

In categorical terms, let $\mc C[I\Sigma_1]$ be the syntactic category of $I\Sigma_1$.\footnote{As mentioned before, the syntactic category construction works for any first-order theory.} Notice that whether we add \emph{PrimRec} as functions symbols to $I\Sigma_1$ together with their axiomatisation or not, the syntactic category $\mc C[I\Sigma_1]$ will be \emph{equivalent}, because this is a definitional extension of $I\Sigma_1$.

There is a natural embedding of $\mc C[\T]$ into $\mc C[I\Sigma_1]$, basically by sending each formula to itself. As a first consequence, \Cref{thm:conservative} implies the following result:

\begin{lemma}\label{lem:transfaith}
  The embedding $\mc C[\T] \to \mc C[I\Sigma_1]$ is faithful.
\end{lemma}
\begin{proof}
  Consider two maps $\theta,\sigma : \varphi \to \psi$ in $\mc C[\T]$. If they are distinct, then $\theta$ and $\sigma$ are not provably equivalent in $\T$. By \Cref{thm:conservative}, they are also not provably equivalent in $I\Sigma_1$, thus they are distinct in $\mc C[I\Sigma_1]$ as well.
\end{proof}

We can in fact characterise the exact objects and morphisms within $\mc C[I\Sigma_1]$ that lies in the image of this embedding. Let $\mc C[I\Sigma_1]_{\Sigma_1}$ be the $\Sigma_1$-subcategory of $\mc C[I\Sigma_1]$, viz. the subcategory consisting of $\Sigma_1$-formulas as objects and $\Sigma_1$-morphisms as morphisms. Notice that this is a well-defined subcategory, because composition of two $\Sigma_1$-morphism by definition is still $\Sigma_1$. The two categories $\mc C[\T]$ and $\mc C[I\Sigma_1]_{\Sigma_1}$ are equivalent:



\begin{corollary}\label{cor:tsigonefrag}
  The inclusion $\mc C[\T] \to \mc C[I\Sigma_1]$ becomes an equivalence when restricting the codomain to $\mc C[I\Sigma_1]_{\Sigma_1}$. 
\end{corollary}
\begin{proof}
  Since objects in $\mc C[\T]$ and $\mc C[I\Sigma_1]_{\Sigma_1}$ are all exactly the $\Sigma_1$-formulas, this is essentially surjective. The conservativity result in \Cref{lem:transfaith} also implies $\mc C[\T] \to \mc C[I\Sigma_1]_{\Sigma_1}$ is fully faithful, since the morphisms are provably total functions defined by $\Sigma_1$-formulas. 
\end{proof}

\Cref{cor:tsigonefrag}, combined with \Cref{cor:ptft}, then suffices to imply the result mentioned in \Cref{subsec:application}, because provably total recursive functions in $I\Sigma_1$ by definition lie in $\mc C[I\Sigma_1]_{\Sigma_1}$, which is equivalently a morphism in $\mc C[\T]$.

\section{Further Proof-Theoretic Properties of Coherent Arithmetic}\label{sec:furtherp}

The initiality result stated in \Cref{thm:ctinitial} also has other applications. As we have seen, characterisation of provably total recursive functions in $\T$ is only an easy consequence of this fact. In this section, we use the initiality theorem to establish further proof-theoretic properties of $\T$.

Our main technical tool is the so-called \emph{Artin glueing} from category theory, which generally applies to a large classes of initial models of certain types of categories; cf.~\citet{aurelio1995glue}. From a type-theoretic perspective, the glueing argument is equivalent to the Tait computability method; cf.~\citet{coquand1997intuitionistic}. Through the lens of category theory, essentially the same argument can be applied to logic of arithmetic.

Given any pr-coherent category $(\mc C,N)$, there is a global section functor
\[ \Gamma : \mc C \to \Set, \]
sending each object $X$ in $\mc C$ to the set of global elements $\mc C(1,X)$ of $X$, which preserves all limits. We can then glue $\mc C$ with $\Set$ along $\Gamma$, and the resulting category is usually called the \emph{Freyd cover} of $\mc C$:

\begin{definition}[Freyd cover]\label{def:freyd}
  The Freyd cover of a pr-coherent category $\mc C$, denoted as $\pf{\mc C}$, is the comma category $\Set \cv \Gamma$ defined as follows:
  \begin{itemize}
  \item Objects: Tuples $(A,X,f)$, where $A$ is a set, $X$ is an object in $\mc C$, and $f : A \to \Gamma X$ is a function.
  \item Morphisms: A morphism from $(A,X,f)$ to $(B,Y,g)$ is a pair $(u,\theta)$, where $u : A \to B$ a function and $\theta : X \to Y$ a morphism in $\mc C$, such that the following diagram commutes,
  \[
  \begin{tikzcd}
    A \ar[d, "u"'] \ar[r, "f"] & \Gamma X \ar[d, "{\Gamma \theta}"] \\
    B \ar[r, "g"'] & \Gamma Y
  \end{tikzcd}
  \]
  \end{itemize}
\end{definition}

There is an evident projection functor $p : \pf{\mc C} \to \mc C$, sending $(A,X,f)$ to $X$ and $(u,\theta)$ to $\theta$. Similarly, there is another projection $q : \pf{\mc C} \to \Set$. The following result is well-known; cf.~\citet{Moerdijk1983}:

\begin{proposition}\label{prop:freydcover}
  Given a pr-coherent category $\mc C$, its Freyd cover $\pf{\mc C}$ is also a pr-coherent category, and the projections $p,q$ are both pr-coherent functor.
\end{proposition}

Now let us take the initial pr-coherent category $\mc C[\T]$. By initiality, there will also be a unique $R : \mc C[\T] \to \pf{\mc C[\T]}$ making the following diagram commute,
\[
\begin{tikzcd}
  & \pf{\mc C[\T]} \ar[d, "p"] \\
  \mc C[\T] \ar[ur, "R"] \ar[r, equal] & \mc C[\T]
\end{tikzcd}
\]
This implies that $RX = (TX,X,\alpha_X : TX \to \Gamma X)$, where $T : \mc C[\T] \to \Set$ is the composition $q \circ R$. The family of maps $\alpha$ actually consists of a natural transformation $\alpha : T \nt \Gamma$, since for any map $\theta : X \to Y$, $R\theta$ will be a map in $\pf{\mc C[\T]}$, and this means the following diagram must commute,
\[
\begin{tikzcd}
  TX \ar[d, "{T\theta}"'] \ar[r, "{\alpha_X}"] & \Gamma X \ar[d, "{\Gamma\theta}"] \\
  TY \ar[r, "{\alpha_Y}"'] & \Gamma Y
\end{tikzcd}
\]

Notice that $T$ is the composition of two pr-coherent functors, thus itself must be pr-coherent. This implies that it is indeed the unique functor from $\mc C[\T]$ to $\Set$, hence sends each $\varphi(\ov x)$ to its canonical interpretation $\N[\varphi(\ov x)]$. The existence of such a natural transformation $\alpha$ already implies the following result:

\begin{theorem}[Truth and provability coincide]\label{thm:trpreq}
  For any sentence $\varphi$ in $\T$, it is provable iff it is true.\footnote{Here \emph{truth} as usual refers to the validity in the standard model.} 
\end{theorem}
\begin{proof}
  If $\varphi$ is true, it follows that $T\varphi = \N[\varphi] = 1$ is the singleton set. Hence, the natural transformation $\alpha_{\varphi}$ gives us some element in $\Gamma\varphi$. By the definition of $\mc C[\T]$, for a sentence $\varphi$, there is a morphism from $\top$ to $\varphi$ iff $\top \vdash \varphi$ is provable in $\T$.
\end{proof}

\begin{remark}
  The fact that any true $\Sigma_1$-sentence is also provable in $I\Sigma_1$ is usually referred to as the \emph{$\Sigma_1$-completeness} of $I\Sigma_1$. This is also a classical result in proof theory, but proven usually by induction on the complexity of formulas. Our proof relies on the natural transformation $\alpha : T \nt \Gamma$, whose existence is guaranteed by the pure structural reason of $\mc C[\T]$ being the \emph{initial} pr-coherent category. It serves as certain \emph{algorithm} that extracts information from the \emph{truth} of a sentence, converting it to a \emph{proof} of that sentence in $\T$. For the reader familiar with type theory, this is indeed the incarnation of \emph{logical relations} in Tait computability methods, adapted to the context of arithmetic. As mentioned at the beginning of this section, Tait computability are widely used in type theory, and recently there has been tremendous success in applying its categorical counterpart, viz. Artin glueing, to the study of complex systems of type theories; e.g.~\citet{sterling2021first}. We hope to at least show the possibility of applying similar methods in the context of proof theory of arithmetic.
\end{remark}

\begin{remark}\label{rem:incom}
  Notice that, although \Cref{thm:trpreq} implies that there are no true but unprovable sentence in $\T$, there could still be \emph{false but irrefutable} sentences in $\T$. In particular, G\"odel's first and second incompleteness theorems still applies to $\T$ in the following sense: There exists a sentence $\varphi$ in $\T$ that is neither provable nor refutable, i.e. both $\top \vdash \varphi$ and $\varphi \vdash \bot$ are \emph{un}provable in $\T$, and the sentence $\mathrm{Incon}_{\T}$, expressing the \emph{in}consistency of $\T$, will be such an example.\footnote{Since there is no negation in $\T$, the consistency of $\T$ as \emph{not} existing a proof of $\bot$ isn't directly formalisable in $\T$, but we can use the \emph{sequent} $\mathrm{Incon}_{\T} \vdash \bot$ to represent the consistency of $\T$.} We will say more about this in \Cref{sec:future}.
\end{remark}

The above theorem has lots of consequences. Firstly, $\T$ proves all the true equality between closed terms. It also implies that $\T$ has the disjunction and existence properties:

\begin{corollary}[Disjunction property of $\T$]\label{cor:disjct}
  For any two sentences $\varphi,\psi$, if $\T$ proves their disjunction,
  \[ \top \vdash \varphi \vee \psi, \]
  then either $\T$ proves $\top \vdash \varphi$, or $\T$ proves $\top \vdash \psi$.
\end{corollary}
\begin{proof}
  $\T$ proves $\varphi \vee \psi$ implies either $\varphi$ or $\psi$ is true, hence at least one of them is provable in $\T$.
\end{proof}

\begin{corollary}[Existence property of $\T$]\label{cor:cant}
  For any formula $\varphi(x)$ in $\T$, if $\T$ proves its existence,
  \[ \top \vdash \exists x\varphi(x), \]
  then there exists some $n\in\N$ that
  \[ \top \vdash \varphi(\ov n). \]
\end{corollary}
\begin{proof}
  Again, if $\T$ proves $\exists x\varphi(x)$, then $\varphi(\ov n)$ is true for some $n\in\N$, hence $\varphi(\ov n)$ will be provable for such $n$.
\end{proof}

\section{Conclusion and future directions}\label{sec:future}

From a semantic perspective, we have carefully studied the internal structure of coherent categories equipped with a PNO, and shown that they support induction rules and construction of bounded universal quantifications of $\Sigma_1$-subobjects. We have also classified the definable functions in the initial pr-coherent category by constructing a pr-coherent category $\PriM$ of primitive recursive functions between recursively enumerable sets.

From a syntactic perspective, we have constructed a coherent theory of arithmetic $\T$, and shown its syntactic category is the initial pr-coherent category. As an application, we have provided a structural proof of the classification of strongly $\Sigma_1$-representable functions in $I\Sigma_1$, which is a classical result in the historical development of proof theory. Other constructive properties of the $\Sigma_1$-fragment of $I\Sigma_1$ also follows from this initiality statement, by using the glueing argument.

At the end of this paper, we also discuss some future directions and further questions naturally arise in this paper:

\subsection{Other arithmetic theories}

As we have mentioned in the introduction, most of the categorical analysis of computability in the literature works in at least Cartesian closed categories where higher types exist. However, many traditional theorems in proof theory cannot be derived in such a framework, because the syntactic categories of these arithmetic theories will not be Cartesian closed. We believe putting more efforts in investigating natural numbers object in weaker categorical context will benefit both categorical recursion theory and traditional proof theory. For instance, could the result in~\citet{buss1986bounded} on the correspondence between bounded arithmetic and polynomial time computable functions be recovered in a structural context? We leave this for future work.

\subsection{Incompleteness theorems}

We have slightly touched upon the incompleteness theorems w.r.t. our coherent theory $\T$ in \Cref{rem:incom}. Since $\T$ lacks negation, the usual construction of a self-referential sentence stating ``I am not provable'' will not be available in $\T$ to show its incompleteness. In a future work, we plan to develop a general framework based on categorical logic to treat the two incompleteness theorems for coherent theories in general, where we may lack negation, implication, and universal quantifier. This will in particular implies that the coherent theory of arithmetic $\T$ defined in this paper is incomplete, and it \emph{cannot refute its own inconsistency}.\footnote{Again, the second incompleteness theorem is interpreted in this way because $\T$ lacks negation, and cannot directly formalise its own consistency.}

\subsection{Comparison with arithmetic universes}\label{subsec:arithuni}

In fact, a categorical approach to incompleteness was proposed and developed by Andr\'e Joyal in the 1970s through a series of unpublished notes and lectures; cf. the much later abstract~\citep{joyal2005godel}. This work is based on a notion of \emph{arithmetic universe}, which are pretopoi with parametrised lists objects. Further developments along this line include~\citet{maietti2003joyal,maietti2010joyal,RN543}, and more recently Joyal's original work has been fully written out in~\citet{van2020g}. 

The categorical framework of arithmetic universe is similar to the pr-coherent categories considered in this paper, where they are both categories with enough structures to interpret coherent arithmetic, but lack higher function types. A priori, our assumption is weaker than an arithmetic universe: We work with coherent categories instead of pretopoi, and we only ask for a PNO, or equivalently a parametrised list object over the terminal object, instead of all objects.

However, it has been suggested to us by an anonymous referee that the initial pr-coherent category constructed in this paper is possible to have a close connection with the initial arithmetic universe constructed in~\citet{van2020g}. In fact, our category $\mc C[\T]$ does have coproducts, since we can remap two formulas $\varphi,\psi$ for them to only consist of even and odd numbers, respectively. With care, one can also see from the encoding of finite lists given in \Cref{sec:init} that $\mc C[\T]$ will also have parametrised list object. Hence, $\mc C[\T]$ will be a positive coherent category with parametrised list objects, and the only structure of an arithmetic universe possibly missing in $\mc C[\T]$ are effective quotients of equivalence relations. $\mc C[\T]$ indeed have some quotients of equivalence relations: If $R(x,y)$ is a \emph{complemented} equivalence relation on $\varphi(x)$, with complement $\qsi R(x,y)$, then the quotient $\varphi/R$ can be evidently constructed as the following formula,
\[ \varphi/R(x) :\equiv \varphi(x) \wedge \buq yx \qsi R(y,x). \]
However, we fail to see how to construct general quotients. Thus, it would be interesting to investigate whether $\mc C[\T]$ is a pretopos, and if not, whether the effective completion of $\mc C[\T]$ (cf.~\citet[A3.3.10]{johnstone2002sketches}) coincides with the initial arithmetic universe.

\subsection{Categorical logic and arithmetic}

Finally, we want to emphasise the perspective of categorical logic. One of the important message from categorical logic is that there is almost an \emph{equivalence} between theories in some fragment of logic with certain kinds of categories; cf.~\citet[D1.4]{johnstone2002sketches}. This perspective on viewing theories as categories allows one to state and prove the initiality result for $\T$, and furthermore to provide a categorical analysis of the proof-theoretic properties of $\T$. For a proof-theorist, category theory in this paper may be viewed as a language that organises different pieces of arguments in proof theory of arithmetic into a structured narrative. However, we expect much more applications of categorical logic, and topos theory in particular, to the logical study of arithmetic.

\section*{Acknowledgement}

We want to thank Lev Beklemishev for reading a first draft of this paper, and for presenting insightful questions and useful suggestions to us. We would like to thank Simon Henry for pointing out an error in an early version of this paper. We would also like to thank a first anonymous referee of this paper for pointing to us important references that were omitted. We also thank a second anonymous referee, which points us to more relevant references, and whose comments directly leads to the addition of \Cref{rem:tiscoherent} filling a significant gap of the paper, and of the discussions in \Cref{subsec:arithuni}.

\bibliographystyle{elsarticle-harv} 
\bibliography{mybib}


\appendix

\section{Complete Proof of \Cref{prop:ctprcoh}}\label{app}

We first show that $\rec_{\gamma,\theta}$ respects the domain and codomain,
\[ \rec_{\gamma,\theta}(x,n,y) \vdash_{x,n,y} \varphi(x) \wedge \psi(y). \]
\begin{itemize}
\item If $n = \ms 0$, then $\rec_{\gamma,\theta}(x,n,y)$ implies $\gamma(x,y)$, which implies $\varphi(x) \wedge \psi(y)$.
\item If $\ms 0 < n$, then $n = \ms sm$. This way, $\rec_{\gamma,\theta}(x,n,y)$ will imply $\exists l(\gamma(x,(l)_{\ms 0}))$, which implies $\varphi(x)$. It also implies $\exists l(\buq un \theta((l)_u,(l)_{\ms su}) \wedge (l)_n = y)$, and this implies $\exists l(\theta((l)_m,(l)_n) \wedge (l)_n = y)$, hence implies $\psi(y)$.
\end{itemize}

The uniqueness of value of $\rec_{\gamma,\theta}$ is the following sequent,
\[ \rec_{\gamma,\theta}(x,n,y) \wedge \rec_{\gamma,\theta}(x,n,z) \vdash_{x,n,y,z} y = z. \]
It is easy to show by a case distinction:
\begin{itemize}
\item When $n$ is $\ms 0$, $\rec_{\gamma,\theta}(x,\ms 0,y) \vdash_{x,y} \gamma(x,y)$, similarly $\rec_{\gamma,\theta}(x,\ms 0,z) \vdash_{x,z} \gamma(x,z)$. Then $y = z$ follows from uniqueness of $\gamma$.
\item When $\ms 0<n$, we may then assume that we have $l,k$ that encodes $\ms sn$-steps of computation,
  \begin{align*}
    &\lh(l) = \ms sn \wedge \gamma(x,(l)_{\ms 0}) \wedge \buq un\theta((l)_u,(l)_{\ms su}) \wedge (l)_n=y, \\
    &\lh(k) = \ms sn \wedge \gamma(x,(k)_{\ms 0}) \wedge \buq un\theta((k)_u,(k)_{\ms su}) \wedge (k)_n=z.
  \end{align*}
  We need to show that $\buq un((l)_u = (k)_u)$, but this can be done by an easy induction on $u$: For the base case $u = \ms 0$, $(l)_{\ms 0} = (k)_{\ms 0}$ again by uniqueness of $\gamma$. For inductive case, if $u < n$, $(l)_{\ms su} = (k)_{\ms su}$ follows from the uniqueness of $\theta$ plus the induction hypothesis $(l)_u = (k)_u$. It follows that $(l)_n = (k)_n$, because there exists some $m$ that $n = \ms sm$ and we have $(l)_m = (k)_m$, and $\theta((l)_m,(l)_n)$ and $\theta((k)_m,(k)_n)$.
\end{itemize}

For the existence of value of $\rec_{\gamma,\theta}$, we need to show
\[ \varphi(x) \vdash_{x,n} \exists y\rec_{\gamma,\theta}(x,n,y), \]
and it can be easily proved by induction on $n$:
\begin{itemize}
\item Base case: For $n = \ms 0$, this follows from the existence of value $\gamma$, and the primitive recursive function of converting any number $z$ to (the code of) a list $\pair{z}$ of length one containing $z$.
\item Inductive case: We need to prove the following sequent,
  \[ \varphi(x) \wedge \exists y\rec_{\gamma,\theta}(x,n,y) \vdash_{x,n} \exists z\rec_{\gamma,\theta}(x,\ms sn,z). \]
  Suppose now we have
  \[ \lh(l) = \ms sn \wedge \gamma(x,(l)_{\ms 0}) \wedge \buq un\theta((l)_u,(l)_{\ms su}) \wedge (l)_n=y. \]
  We may construct $z$ as the unique value of $\theta(y,z)$, and construct $k$ as $l \star \pair{z}$, where $\star$ is a primitive recursive term denoting the concatenation of sequences. We should then be able to verify
  \[ \lh(k) = \ms s\ms sn \wedge \gamma(x,(k)_{\ms 0}) \wedge \buq u{\ms sn}\theta((k)_u,(k)_{\ms su}) \wedge (l)_n=z, \]
  hence concluding $\exists z\rec_{\gamma,\theta}(x,\ms sn,z)$.
\end{itemize}

We prove the commutativity of the following diagram,
\[
\begin{tikzcd}
  \varphi(x) \ar[dr, "{\gamma}"'] \ar[r, "{\pair{\id,\ms 0}}"] & \varphi(x) \times N \ar[d, "{\rec_{\gamma,\theta}}"] & \varphi(x) \times N \ar[l, "{\id \times \ms s}"'] \ar[d, "{\rec_{\gamma,\theta}}"'] \\
  & \psi(y) & \psi(y) \ar[l]^{\theta}
\end{tikzcd}
\]
The commutativity of the triangle amounts to saying that the following sequent is provable,
\[ \gamma(x,y) \dashv\vdash_{x,y} \rec_{\gamma,\theta}(x,\ms 0,y), \]
which should be immediate from the definition of $\rec_{\gamma,\theta}$. The commutativity of the square amounts to saying that
\[ \rec_{\gamma,\theta}(x,\ms sn,y) \dashv\vdash_{x,n,y} \exists z(\rec_{\gamma,\theta}(x,n,z) \wedge \theta(z,y)). \]
From left to right, if there is an $l$ encoding a computation up to $\ms sn$, then we can extract the value of $z$ as $(l)_n$, and obtain another list $\ms{droplast}(l)$ dropping the last entry of $l$, and these should witness $\rec_{\gamma,\theta}(x,n,z)$ and $\theta(z,y)$. The function $\ms{droplast}$ is again a primitive recursive term in $\T$. From right to left, we do the reverse process. Given $\rec_{\gamma,\theta}(x,n,z)$ with a witnessing list $l$, and $\theta(z,y)$, we construct $l \star \pair{y}$, and this should verify $\rec_{\gamma,\theta}(x,\ms sn,y)$.

Finally, we prove the morphism $\rec_{\gamma,\theta}$ is unique. Suppose we have another morphism $\sigma(x,n,y)$ from $\varphi(x) \times N$ to $\psi(y)$ making the above diagram commute. We need to show that
\[ \sigma(x,n,y) \dashv\vdash_{x,n,y} \rec_{\gamma,\theta}(x,n,y). \]
From left to right, it relies on the fact that for definable functions in $\T$, we can construct a list of its values up to an arbitrary number,
\[ \varphi(x) \vdash_{x,n} \exists l(\lh(l) = \ms sn \wedge \buq u{\ms sn}\sigma(x,u,(l)_u)). \]
We may then verify that the same list can be used to construct values for the morphism $\rec_{\gamma,\theta}$,
\begin{align*}
  &\lh(l) = \ms sn \wedge \buq u{\ms sn}\sigma(x,u,(l)_u) \\
  \vdash_{x,l,n}\ &\lh(l) = \ms sn \wedge \gamma(x,(l)_{\ms 0}) \wedge \buq un\theta((l)_u,(l)_{\ms su}).
\end{align*}
Essentially, we need to show that
\[ \buq u{\ms sn}\sigma(x,u,(l)_u) \vdash_{x,l,n} \gamma(x,(l)_{\ms 0}), \]
and that
\[ \buq u{\ms sn}\sigma(x,u,(l)_u) \wedge i < n \vdash_{x,l,n,i} \theta((l)_i,(l)_{\ms si}). \]
These two properties should then be immediate from the fact that $\sigma$ makes the above diagram commute.

The right to left direction is completely similar. The trick it to prove the following stronger result by induction on $i$,
\[ \lh(l) = \ms sn \wedge \gamma(x,(l)_{\ms 0}) \wedge \buq un\theta((l)_u,(l)_{\ms su}) \vdash_{x,l,n,y,i} n < i \vee \sigma(x,i,(l)_i). \]
\begin{itemize}
\item Base case: When $i = \ms 0$, $\gamma(x,(l)_{\ms 0})$ implies $\sigma(x,\ms 0,(l)_{\ms 0})$.
\item Inductive case: When $i = \ms sj$, induction hypothesis gives us $\sigma(x,j,(l)_j)$, then $\theta((l)_j,(l)_i)$ implies $\sigma(x,i,(l)_i)$. 
\end{itemize}
The above sequent in particular implies $\sigma(x,n,(l)_n)$, thus we would have
\[ \rec_{\gamma,\theta}(x,n,y) \vdash_{x,n,y} \sigma(x,n,y). \]
This completes the whole proof.

\end{document}